\documentclass[oneside,english]{amsart}
\usepackage{ae,aecompl}
\usepackage[T1]{fontenc}
\usepackage[latin9]{inputenc}
\usepackage{geometry}
\usepackage{babel}
\usepackage{amstext}
\usepackage{amsthm}
\usepackage{amssymb}
\usepackage{setspace}
\usepackage[authoryear]{natbib}
\PassOptionsToPackage{normalem}{ulem}
\usepackage{ulem}
\usepackage[unicode=true,pdfusetitle,
 bookmarks=true,bookmarksnumbered=false,bookmarksopen=false,
 breaklinks=false,pdfborder={0 0 1},backref=false,colorlinks=false]
 {hyperref}

\makeatletter
\numberwithin{equation}{section}
\numberwithin{figure}{section}
\theoremstyle{plain}
\newtheorem{thm}{\protect\theoremname}
\theoremstyle{plain}
\newtheorem{fact}[thm]{\protect\factname}
\newenvironment{lyxlist}[1]
	{\begin{list}{}
		{\settowidth{\labelwidth}{#1}
		 \setlength{\leftmargin}{\labelwidth}
		 \addtolength{\leftmargin}{\labelsep}
		 }}
	{\end{list}}
\theoremstyle{plain}
\newtheorem{cor}[thm]{\protect\corollaryname}
\theoremstyle{plain}
\newtheorem{prop}[thm]{\protect\propositionname}
\theoremstyle{remark}
\newtheorem*{claim*}{\protect\claimname}
\theoremstyle{plain}
\newtheorem{lem}[thm]{\protect\lemmaname}
\theoremstyle{definition}
\newtheorem{defn}[thm]{\protect\definitionname}

\usepackage[noload]{qtree}
\usepackage{amsfonts}
\usepackage{bussproofs}
\usepackage{pdflscape}

\usepackage{tikz}
\usepackage{tikz-cd}
\usetikzlibrary{matrix,arrows,decorations.pathmorphing,arrows.meta}
\usetikzlibrary{positioning,calc}

\makeatother

\providecommand{\claimname}{Claim}
\providecommand{\corollaryname}{Corollary}
\providecommand{\definitionname}{Definition}
\providecommand{\factname}{Fact}
\providecommand{\lemmaname}{Lemma}
\providecommand{\propositionname}{Proposition}
\providecommand{\theoremname}{Theorem}

\begin{document}
\title{Internal Categoricity and the Generic Multiverse}
\author{Toby Meadows}
\begin{abstract}
John Steel's theory, $MV$, of the generic multiverse provides a foundation
for mathematics that aims to neutralize the effects of incompleteness
brought on by forcing arguments. Jouko V\"{a}\"{a}n\"{a}nen's development of internal
categoricity arguments provides opportunities to argue that the subject
matter of some theory is, in some sense, determined. This paper investigates
whether $MV$ is internally categorical.
\end{abstract}

\thanks{I'm very grateful to Gabe Goldberg for helping with some of the crucial
claims in this paper. I would like to thank Pen Maddy for helpful
conversations around the underlying philosophy. I also want to thank
Joan Bagaria, Jeremy Heis and John Steel. }
\date{August 20, 2025}

\maketitle
It is well known that second order $ZFC$ is almost categorical. Given
a pair of full models of second order $ZFC$, one of them will be
isomorphic to a rank initial segment of the other and the rank of
that initial segment will be an inaccessible cardinal. It is a pleasing
uniqueness theorem that sits naturally beside the existence theorems
delivered by consistency proofs. Moreover, the moths of philosophy
have been drawn to its tantalizing flame since \citet{Zermelo} and
\citet{KreisalComp}. One might hope to use this result to conclusively
demonstrate that a large chunk of the subject matter of set theory,
including the continuum hypothesis, is definitively fixed. These hopes
have largely been dashed on the rocks of circularity, however, more
recently Jouko V\"{a}\"{a}n\"{a}nen has formulated an intriguing notion of \emph{internal
categoricity }that seems to avoid these problems and perhaps deliver
opportunities for more promising philosophical arguments \citep{VaananenTrIntCat,MaddyVaananenCat}.
Internal categoricity will be the focal concept of this paper. 

As to this paper's target, we consider a theory that looks to be motivated
by ideas that oppose the fixity of subject matter suggested by categoricity.
Since the advent of Cohen's forcing technique and the realization
that the continuum hypothesis is independent of $ZFC$, there has
been a recurring thought that at least some of the incompleteness
of $ZFC$ is more of a feature than a bug \citep{MostowskiRRST}.
This motivates the idea that there is a plurality of different universes
wrought through forcing that might all be taken seriously as alternatives
upon a somewhat level playing field. W. Hugh Woodin brought this idea
to life, albeit in order to bury it, along with its provocative name:
the \emph{generic} \emph{multiverse} \citep{WoodinGM,WoodinROI,WoodinSTAR}.
This is a multiverse of universes linked together by sequences of
generic extensions and refinements. More recently, John Steel gave
this idea a more solid foundational footing in the form of a computably
axiomatizable theory, $MV$ \citep{SteelGP}. $MV$ will be the main
target of this paper.

Putting these ideas together, we have an interesting tension. On one
hand, we have an intriguing theorem template that might rigidify the
subject matter of set theory; and on the other, we have a theory that
appears to embrace a radically pluralistic vision of what set theory
is about. To bring this contrast to a point, the categoricity arguments
seem to tell us that the continuum hypothesis has a definitive answer,
while the generic multiverse appears to embrace its indeterminacy.
In spite of all these appearances, this paper has a simple, but --
I think -- surprising goal. We are going to prove that $MV$ is internally
categorical!

Before we get too carried away, I'll say something even stranger and
then I'll try to clarify the underlying picture. So first, We are
also going to prove that $MV$ is \emph{not} internally categorical.
Now it sounds like we have a contradiction, but the devil is in the
details; and in particular, the axioms we use. Just as we have with
$PA$ and $ZFC$, the internal categoricity of $MV$ can be made to
hold or not hold depending on the axioms we use to generate that theory.\footnote{See, for example, Theorems 47 and 52 in \citep{EnayatLelykFOCat}.}
Where does this leave the philosophical preambling above? We'll defer
some discussion of that until then end, but I will say that these
results nicely illustrate that internal categoricity is a subtle matter.
Moreover, this subtlety warrants further investigation and $MV$ presents
an intriguing case study for that project. 

The paper is organized as follows. In Section \ref{sec:What-is-internal},
we provide a brief overview of internal categoricity and some of the
philosophical arguments it is entangled with. Section \ref{sec:What-is-?}
provides an overview of Steel's $MV$ that includes a rough classification
of $MV$'s models. In Section \ref{sec:When--is}, we prove that $MV$
is not internally categorical and then in Section \ref{sec:When--is-1},
we prove that an alternative axiomatization of $MV$ is internally
categorical. The paper then closes in Section \ref{sec:Discussion}
with some discussion of these results.

\section{What is internal categoricity?\label{sec:What-is-internal}}

We know from Dedekind that second order $PA$ is categorical and from
Zermelo and Shepherdson that second order $ZFC$ is quasi-categorical.
More formally, 
\begin{thm}
\label{thm:DedZerm}(1) \citep{Dedekind} If $\mathcal{M}$ and $\mathcal{N}$
are full models of second order $PA$, then $\mathcal{M}\cong\mathcal{N}$.

(2) \citep{Zermelo,ShepInI} If $\mathcal{M}$ and $\mathcal{N}$
are full models of second order $ZFC$, then either: $\mathcal{M}$
is isomorphic to $V_{\kappa}^{\mathcal{N}}$ where $\kappa$ is inaccessible
in $\mathcal{N}$; or $\mathcal{N}$ is isomorphic to $V_{\lambda}^{\mathcal{M}}$
where $\lambda$ is inaccessible in $\mathcal{M}$. 
\end{thm}

With regard to (1), we have an intriguing uniqueness result. To put
some context around this, let's note that for most us, it's obviously
desirable that, despite the weaknesses of first order logic in pinning
down structures, there should be just one \emph{intended} model of
arithmetic (up to isomorphism).\footnote{As we progress further in this paper, I'll frequently omit the phrase
``up to isomorphism'' in the hopes that this will keep the words
flowing smoothly and that the reader will charitably discern my true
intentions. } It seems natural then to wonder if Dedekind's result \emph{provides}
\emph{evidence} that this is the case. There is a lot of philosophy
to unpack in such a conjecture, but fortunately this has been thoroughly
discussed before \citep{KreisalComp,ParsonNatNum,ButtonWalshCat,MaddyVaananenCat}.
As such, we'll satisfy ourselves by briefly considering just one compelling
objection, which I'll motivate with a question: ``If I were worried
about whether there was a unique intended model of arithmetic, then
surely I should be even more worried about full second order models
of $PA$ which presuppose that it makes sense to talk about the unique
powerset of the natural numbers?'' It has been suggested many times
that the conjecture above makes use of a kind of circular reasoning
\citep{WestCont,HamMult,MeadowsCat,KoellnerHAM,MaddyVaananenCat}.
In order justify our thought that there is a unique model of arithmetic,
we appeal to the uniqueness of a much more complex structure. As such,
the idea that Dedekind's result provides evidence for our intuitions
about arithmetic, looks to be on shaky ground. None of this is to
say that there isn't a unique model of arithmetic, just that Dedekind's
result doesn't provide compelling evidence for that claim. 

With regard to (2), we get closer to the target of this paper and
analogous remarks apply. Of course, we don't always get a genuine
isomorphism between full models of second order $ZFC$. So there is
less temptation to think we have evidence that the subject matter
of set theory has been pinned down. But any pair of such models still
have a great deal in common. In particular and saliently, any two
full models of second order $ZFC$ must agree on their evaluation
of the continuum hypothesis. Thus, we might be tempted to think that
Zermelo's result provide us with some evidence that the continuum
hypothesis has a definitive answer. As with arguments making use of
Dedekind's theorem, worries about circularity also emerge in this
arena. Once again, the problem comes from the use of full second order
models and their free wheeling use of the powerset operation. 

Much more recently, V\"{a}\"{a}n\"{a}nen has offered a new perspective on categoricity
theorems that makes no use of second order logic and thus, defuses
the circularity worries discussed above. He calls this \emph{internal
categoricity} \citep{VaananenTrIntCat,VaanZerm}.\footnote{Strictly speaking, V\"{a}\"{a}n\"{a}nen considers two kinds of internal categoricity.
The first is still rooted in second order logic, although uses Henkin
models instead of full models of second order logic. The second only
uses first order logic and will be our focus in this paper. This is
also the usage taken up in \citep{EnayatLelykFOCat}.} We start by considering an expansion of the ordinary language of
set theory where we let $\mathcal{L}(\in_{0},\in_{1})$ be the language
with a pair of 2-place relation symbols. We then let $ZFC(\in_{0},\in_{1})$
be the theory that has: an $\in_{0}$ and $\in_{1}$ version of every
axiom of $ZFC$; and has $\in_{0}$ and $\in_{1}$ versions of the
Replacement and Separation Schemata where the formulae used may make
use of both $\in_{0}$ and $\in_{1}$.\footnote{My notation is cosmetically different from that of \citet{VaananenTrIntCat}.
I've done this so that I can notate things a little more compactly
on the page, but do refer to \citet{VaananenTrIntCat,VaanZerm,MaddyVaananenCat}
for the official notation. Moreover, \citet{EnayatLelykFOCat} provide
an excellent notation system for generalizing these ideas to other
theories. Given that our focus is quite narrow in this paper, we'll
avoid this and just remind the reader of our meaning when required. } We then obtain the following. 
\begin{thm}
\citep{VaanZerm} If $\mathcal{M}=\langle M,\in_{0}^{\mathcal{M}},\in_{1}^{\mathcal{M}}\rangle$
is a model of $ZFC(\in_{0},\in_{1})$, then there is an isomorphism
\[
\sigma:\langle M,\in_{0}^{\mathcal{M}}\rangle\cong\langle M,\in_{1}^{\mathcal{M}}\rangle
\]
where $\sigma$ is uniformly definable over $ZFC(\in_{0},\in_{1})$. 
\end{thm}

Some remarks about the relationship with Zermelo's theorem are warranted.
With Zermelo, we learn something about a pair of models, while with
V\"{a}\"{a}n\"{a}nen we are just considering a single model. The interesting thing
is that this model is -- in a sense -- two models in disguise. Moreover,
those two models are deeply entangled through the mashup theory $ZFC(\in_{0},\in_{1})$.
In particular, they are allowed to make us of each other's languages
in their respective Replacement and Separation Schemata. Thus, and
helpfully for the proof of the theorem above, each model can deliver
the transitive collapse of the other.

But why should we allow this sharing of vocabulary between their Replacement
and Separation Schemata? From a mathematical point of view, there
is nothing stopping us. But from a more philosophical perspective,
there is an intriguing argument in support of this move. If we consider
the intention behind our use of a schema to articulate the idea behind
Replacement, we might say that it is intended to be a first order
approximation of the second order idea that the pointwise image of
\emph{any} class function through a set should itself be a set of
outputs.\footnote{See \citet{ParsonNatNum} and \citet{FieldDeflCons} for thorough
arguments for this attitude toward schemata in the context of arithmetic.
See Section 4.3 of \citep{MaddyVaananenCat} for a nuanced discussion
of the relationship Parson's thought experiment and internal categoricity.
One might worry that this talk of approximating a second order concept
with a schema doesn't take us far enough away from second order logic.
See page 30 of \citep{MaddyVaananenCat} for further discussion of
this point. We shall raise this again in the Discussion section.} Thus, if we expand our language in such a way that could allow us
to define new class functions, these class functions ought to also
be subject to the closure conditions required by Replacement.

What should we make of this? To motivate the milieu behind such results,
we might consider a thought experiment based on ideas from interpretability
that is similar in spirit if not letter to \citet{ParsonNatNum}.
Suppose that you use a theory $DLO(<)$ of dense linear orders that
is based on a strict $<$ relation, and I use a theory $DLO(\leq)$
based on non-strict relation $\leq$. There is, of course, no substantive
debate to be had between us. But how might we demonstrate this? Given
that you are working in $\mathcal{L}(<)$, you might move to an expansion
$\mathcal{L}(<,\leq)$ that adds my non-strict $\leq$. You might
then accommodate my thoughts about the $\leq$ by extending your theory
to be $DLO(<)\cup DLO(\leq)$. This puts us in a position to model
how you might think about my $\leq$ in terms of your $<$. At this
point, we know nothing about that, but you might be suddenly struck
by inspiration and observe that you can add an axiom to $DLO(<)\cup DLO(\leq)$
that defines $\leq$ in terms of $<$ by saying that $x\leq y$ iff
$x<y$ or $x=y$. You might then prove that the resulting theory remains
consistent. Moreover, I might then do the analogous thing but in the
other direction. If we both succeed then it can be seen that our initial
theories were very closely related; in particular $DLO(<)$ and $DLO(\leq)$
are definitionally equivalent.\footnote{See \citep{LefeverDefEqNonDisjLang} for an excellent overview of
different characterizations of definitional equivalence.} This is a simple illustration of a pivotal concept based on the theory
of interpretability. Results like this are a cornerstone of contemporary
set theory.

I think there is a analogy to be had between the example above and
a way of applying V\"{a}\"{a}n\"{a}nen's theorem. Suppose you and I are both users
of $ZFC$. In this situation, we have no worries about whether our
theories are definitionally equivalent since they are identical. However,
we might wonder if your version of $\in$ has the same meaning as
mine. To explore this question, we might tag our membership relations
to differentiate them and thus, use $\in_{0}$ and $\in_{1}$. We
then might say that you are working $ZFC(\in_{0})$ articulated in
$\mathcal{L}(\in_{0})$ and I am working in $ZFC(\in_{1})$ articulated
in $\mathcal{L}(\in_{1})$. In order to explore the relationship between
our theories you might then expand your language from $\mathcal{L}(\in_{0})$
to $\mathcal{L}(\in_{0},\in_{1})$ and accommodate my ideas about
set theory by working in the theory $ZFC(\in_{0})\cup ZFC(\in_{1})$.
This puts us in a good position to model how you might think about
my $\in_{1}$ in terms of your $\in_{0}$. But at this point, our
setup has nothing to say on this matter. Then, as before, you might
be struck by inspiration and adopt a \emph{Parsonian} attitude to
the Replacement and Separation Schemata by permitting formulae of
the total language in both versions of those schemata. This leaves
us in $ZFC(\in_{0},\in_{1})$ where we may finally use V\"{a}\"{a}n\"{a}nen's
theorem to define an isomorphism between my version of the membership
relation and yours. This might give you some reason to think that
what I mean by $\in_{1}$ is very closely related to you mean by $\in_{0}$.
Moreover, it might give you some reason to think that whatever you
think about the truth value of the continuum hypothesis, I should
think the same. The analogy with the dense linear orders is certainly
not perfect, but I think it nicely illustrates the fact that we are
dealing with a common or, at least, overlapping methodology. If we
want to compare a pair of theories, then it makes good sense to consider
ways of combining them. Finally, we note that second order logic has
disappeared from view and so the circularity worries we had above
no longer have a target. Putting all this together, it could seem
that we have good reason to think that you and I should agree on the
truth value of the continuum hypothesis using a methodology common
in set theory that does not fall prey to worries about circularity. 

I don't plan to precisely formulate arguments about the speculative
reasoning above. Nor do I plan to evaluate their force. I have two
reasons for this. First, I think there is value in merely illustrating
the intuitive pull that surrounds V\"{a}\"{a}n\"{a}nen's theorem in an effort
to flag it for more intense philosophical scrutiny. Second, I think
the results that follow below place a sufficiently large question
mark over such argument sketches that the effort required to understand
what is going on will be substantial and well beyond the scope of
this paper.

\section{What is $MV$?\label{sec:What-is-?}}

The incompleteness of $ZFC$ as evidenced by forcing arguments is
sometimes thought of as particularly deep. By contrast, shallow forms
of incompleteness can be resolved by adopting a stronger theory that
we think is true. For example, although $ZFC$ cannot prove that it
is consistent, if we adopt an axiom saying that there are inaccessible
cardinals then the consistency of $ZFC$ follows almost immediately.
But over at the deep end, the continuum hypothesis can be both forced
to hold and forced to fail. Moreover, no large cardinal axiom can
come to the rescue and decide the matter. The generic multiverse was
introduced as a means of taking such deep incompleteness seriously
as something we might just have to live with. In its simplest form,
a generic multiverse $\mathbb{V}_{M}$ is formed by taking a countable
transitive model $M$ of $ZFC$ and then considering all the models
that can be obtained from $M$ by finite sequences of generic extensions
and refinements \citep{WoodinGM}. We might then live with the incompleteness
of $ZFC$ by only taking seriously those statements that are true
in every world $N$ in $\mathbb{V}_{M}$. In particular, the continuum
hypothesis is relegated to the liminal space between always true and
always false. Woodin describes the motivation behind the generic multiverse
particularly clearly:
\begin{quote}
Is the generic-multiverse position a reasonable one? The refinements
of Cohen's method of \emph{forcing} in the decades since his initial
discovery of the method and the resulting plethora of problems shown
to be unsolvable, have in a practical sense almost compelled one to
adopt the generic-multiverse position.\footnote{I should mention at this point that although Woodin clearly describes
why embracing the generic multiverse might be attractive and understandable,
his goal in that paper is to argue against its adoption.} \citep{WoodinGM}
\end{quote}
If our goal is to take deep incompleteness seriously, then Woodin's
$\mathbb{V}_{M}$ has some limitations. $\mathbb{V}_{M}$ is a model
that we describe, so to speak, from the outside rather than foundational
framework like $ZFC$ than we can use, so to speak, from the inside.
It's still a valuable tool for understanding how the generic multiverse
position works, but it doesn't offer us a way of truly inhabiting
its multiverse. To address this, Steel introduced a computably axiomatizable
theory of the generic multiverse that is designed to be an autonomous
foundation for mathematics, which is on a par with $ZFC$, and within
which we can live with deep incompleteness \citep{SteelGP}.

Steel's theory has received a healthy dose of high quality philosophical
attention since its inception \citep{MeadMadPGMV,Meadows2Args,BAGARIA_TERNULLO_2023,BLUE_2024}.
My goals in this section and the rest of this paper are more logical
in spirit. I want to give an overview of: what $MV$ looks like from
the outside, through a rough classification of its models; and what
$MV$ feels like on the inside, by highlighting some significant consequences
of the theory. My thought is that by delivering a more comprehensive
overview of $MV$, the internal categoricity results below will make
more sense to the reader and perhaps over-optimistically that these
results might make $MV$ more accessible to a wider class of readers. 

\subsection{Notation}

Before we launch into the theory, let's first make some remarks about
notation. In general, I am aiming to follow the notation conventions
used by contemporary set theorists as used in, for example \citep{JechST,KunenST2}
and \citep{foreman2009handbook}. However, I'll highlight a few local
idiosyncrasies that I've taken up in order to make things look a little
smoother on the page. 

First and contrary to the usual conventions in mathematical logic,
a theory will just be a set of sentences that we usually call axioms.
We shall not demand, unless stated otherwise, that a theory is closed
under logical consequence. Given that internal categoricity is sensitive
to the axioms used to generate a theory, this is particularly important.\footnote{We should note that this approach to theories is quite common in the
theory of interpretability. See, for example, \citep{lindstrm2003aspects}
and \citep{EnayatLelykFOCat}.}

We shall, below, be frequently concerned with models of $\mathcal{L}_{\in}=\{\in\}$.
Most often these will be transitive models and for those, we adopt
the usual convention of denoting the model $\langle M,\in\rangle$
by its domain $M$. However, when we come to denote models that have
a non-standard membership relation we shall use the calligraphic font
and write $\mathcal{M}=\langle M,\in_{\mathcal{M}}\rangle$.

In almost all of the many simple forcing arguments that follow we
shall make great use of L\'{e}vy's collapse forcing. Recall that $Col(\omega,X)$
for $X\subseteq Ord$ is the partial order whose conditions are finite
partial functions $p:\omega\times X\rightharpoondown X$ where for
all $n\in\omega$ and $\alpha\in X$, $p(n,\alpha)<\alpha$ if it
is defined. We then order $Col(\omega,X)$ be reverse inclusion. Rather
than write $Col(\omega,X)$ again and again, we shall instead write
$\mathbb{C}_{X}$ where the $\mathbb{C}$ is intended to be mnemonic
for \uline{c}ollapse. We shall also use interval notation for sets
of ordinals so, for example $[\alpha,\beta)=\{\gamma\in\beta\ |\ \gamma\geq\alpha\}$.
If $H$ is $\mathbb{C}_{\beta}$-generic over $V$ for some $\beta\in Ord\cup\{Ord\}$
and $\alpha<\beta$, we write $H\restriction\alpha$ to denote $\{p\in\mathbb{C}_{\alpha}\ |\ p\in H\}$. 

The following theorem of Laver and Woodin plays a pivotal role in
the theory of the generic multiverse. We shall also use the terms
``generic refinement'' and ``ground'' as synonyms below. $U$
is a generic refinement of $W$ if there is some $U$-generic $G$
such that $W=U[G]$.
\begin{fact}
(Laver, Woodin)\label{fact:Laver,-Woodin} Every generic refinement
$U$ of $V$ is uniformly definable in $V$ using a parameter from
$U$. That is, there is a formula $\psi(x,y)$ of $\mathcal{L}_{\in}$
such that if $V=U[G]$ is a forcing extension of $U$ by a $U$-generic
filter $G\subseteq\mathbb{P}\in U$, then there is some $r\in U$
such that\footnote{See \citep{ReitzTGA} for a gentle proof and \citep{WoodinGM} for
something more general. } 
\[
U=\{x\ |\ \psi(x,r)\}.
\]
\end{fact}

Let us fix a formula $\psi(y,x)$ as in the statement of the fact
above. We shall write $x\in\mathsf{W}_{r}$ to mean that $\psi(x,r)$
holds. This is usually just written as $W_{r}$ rather than $\mathsf{W}_{r}$.
However, when talking about the generic multiverse and using the theory,
$MV$, upper case variables like $W$ are in high demand. Thus, we
use the sans serif font $\mathsf{W}$ to avoid confusion. 

\subsection{The $MV$ theory\label{subsec:The--theory}}

We are now ready to describe $MV$. The theory is articulated in the
language $\mathcal{L}_{MV}$ which is a two sorted language extending
$\mathcal{L}_{\in}$ with a new sort for \emph{worlds} beyond the
existing sort for \emph{sets}. Sets can be members of both sets and
worlds, but worlds can be members of neither sets nor worlds.\footnote{Strictly we should have two membership relations: one for when sets
are members of sets; and the other for when sets are members of worlds.
However, using just one symbol lightens our notation load in a helpful
way so we adopt this abuse of notation. } The axioms of $MV$ can then be stated somewhat formally as follows.
We shall write $gen(\mathbb{P},W)$ to denote the set of $\mathbb{P}$-generic
filters of a world $W$. 
\begin{lyxlist}{00.00.0000}
\item [{(World-extensionality)}] $\forall x(x\in W\leftrightarrow x\in U)\to W=U$
\item [{(World-transitivity)}] $x\in W\to x\subseteq W$
\item [{(World-ordinals)}] $\alpha\in Ord^{W}\leftrightarrow\alpha\in Ord^{U}$
\item [{($ZFC$-worlds)}] $\varphi^{W}$ where $\varphi$ is an axiom of
$ZFC$. 
\item [{(Set-capture)}] $\exists W\ x\in V$
\item [{(Extension)}] $\mathbb{P}\in W\exists U\exists g\in U\ (g\in gen(\mathbb{P},W)\wedge U=W[g])^{U}$
\item [{(Refinement)}] $r\in W\exists U\ (U=(\mathsf{W}_{r})^{W}$
\item [{(Amalgamation)}] 
\begin{align*}
 & \exists\mathbb{P}_{0}\in W_{0}\exists\mathbb{P}_{1}\in W_{1}\exists U\exists g_{0},g_{1}\in U\ \\
 & (g_{0}\in gen(\mathbb{P}_{0},W_{0})\wedge g_{1}\in gen(\mathbb{P}_{1},W_{1})\wedge\\
 & W_{0}[g_{0}]=U=W_{1}[g_{1}])^{U}.
\end{align*}
\end{lyxlist}
The first five axioms essentially say that worlds are inner models
of $ZFC$ and that every set $x$ is captured by a world $W$ in the
sense that $x\in W$. The final three axioms breathe life into the
\emph{generic} aspect of this multiverse. Extension tells us that
any generic extension of a world is itself a world. Refinement tells
us that any generic refinement of a world is a world. Amalgamation
tells us that any pair of worlds have a common generic extension. 

Extension and Refinement are straightforwardly motivated. They tell
us that the multiverse is closed under forcing and its inverse. Amalgamation
is a little more complicated.\footnote{See \citet{MeadMadPGMV} for some extensive discussion of this. }
One quite simple-minded motivation for including it is that it ensures
that our multiverse doesn't contain, so to speak, isolated galaxies
since any pair of worlds are subsumed into a larger world. However,
we shall see very soon that Amalgamation also allows us to obtain
particularly elegant models of $MV$ that have a structure which is
very familiar to set theorists. 

It's easy enough to offer these informal explanations but they only
get us so far. In the next section, we'll start thinking more seriously
about models of $MV$, but before we do this, let's ask a natural
question about the consequences of this theory: 
\begin{quote}
Given that $MV$ is a theory of sets and worlds, we can consider its
reduction to the set domain, which gives us a theory in $\mathcal{L}_{\in}$.
What is that theory like?
\end{quote}
It turns out that it has a very nice theory indeed. Informally, we
might say that it satisfies countable set theory. More precisely,
let $ZFC^{-}$ be $ZFC$ articulated using the Collection schema (rather
than Replacement) and a Set Induction schema (rather than Foundation)
where we remove the Powerset axiom.\footnote{Nothing important swings on using Set Induction rather than Foundation.
Like Collection, it just behaves better in weakenings of $ZFC$.} Let $ZFC_{count}^{-}$ be $ZFC^{-}$ with the following additional
axiom: $\forall x\ |x|\leq\omega$. This is the natural theory of
the hereditarily countable sets and it is implied by $MV$.\footnote{Note that Choice is already implied by $\forall x\ |x|\leq\omega$
but we retain the $C$ just to emphasize that Choice is available. }
\begin{thm}
$MV\vdash ZFC_{count}^{-}$.\label{thm:MV|-ZFC_count}
\end{thm}

Before we remark on the proof of this theorem, this is a good moment
to return our attention to the continuum hypothesis and consider how
it should be understood in the context of $MV$. We remarked above
that $MV$ was developed with an eye to taking the independence of
the continuum hypothesis seriously. But how is this achieved? It turns
out that there are two ways of thinking about this. On the first way
of understanding things, we treat each world in the multiverse as
a genuine alternative and we see then that there will be worlds where
the continuum hypothesis is true and worlds where it is false. This
is a theorem of $MV$ that is essentially a consequence of Cohen's
initial work on forcing \citep{Cohen1963-COHTIO-5}. We might call
this the \emph{local perspective}. From this point of view, the cardinality
of the continuum is not determined since it has many possible values
across the multiverse. By contrast, the \emph{global perspective}
takes up the point of view considered in Theorem \ref{thm:MV|-ZFC_count}.
Rather than worrying about whether the continuum hypothesis is true
in one world or false in another, we ask whether it is true in the
underlying, global, universe of sets that sits behind each of the
worlds in the multiverse. The theorem above then tells us that every
set is countable. Unless we wish to contradict Cantor's famous theorem,
this means that there cannot be a set of all subsets of the natural
numbers: $\mathcal{P}(\omega)$ does not exist. Thus, from the global
perspective the continuum hypothesis is simply inexpressible and thus,
arguably enjoys an even greater degree of indeterminacy than is exhibited
locally. I do not wish to put forward a claim that one perspective
is better than the other or that a choice must be made. After all,
they both fit the formal theory quite well. Regardless, it is clear
that -- either way -- the continuum hypothesis has no definitive
answer in $MV$. 

\subsection{The models of $MV$}

Let return to Theorem \ref{thm:MV|-ZFC_count} and its proof. It shouldn't
be too surprising to see that $MV$ implies every set is countable:
Extension ensures that for any cardinal, there will be a world where
it has been collapsed. Since Cantor's theorem tells us that any set
and its powerset have different cardinalities, we also see that the
Powerset Axiom must fail. To show that $ZFC^{-}$ is satisfied, it
will be helpful to consider some existing results about models of
$MV$. Let's begin with some technical preliminaries.

A model of $MV$ is of the form $\mathcal{W}=\langle\mathbb{W}^{\mathcal{W}},\mathbb{D}^{\mathcal{W}},\in^{\mathcal{W}}\rangle$
where $\mathbb{W}^{\mathcal{W}}$ is the world domain, $\mathbb{D}^{\mathcal{W}}$
is the world domain and $\in^{\mathcal{W}}$ is the membership relation.
In \citep{SteelGP} and \citep{MeadMadPGMV}, we are given a particularly
nice class of models for $MV$ that we might regard as its \emph{standard
models}. Let $\mathcal{M}=\langle M,\in^{\mathcal{M}}\rangle$ be
a model of $ZFC$ and suppose that $G$ is $(\mathbb{C}_{Ord})^{\mathcal{M}}$-generic
over $\mathcal{M}$. Then let $\mathcal{M}[G]\supseteq\mathcal{M}$
be the (class) generic extension of $\mathcal{M}$ via $G$.\footnote{If $\mathcal{M}$ isn't transitive, it's a little fiddly to ensure
that $\mathcal{M}\subseteq\mathcal{M}[G]$. A way of doing this is
described in Appendix A.1 of \citep{MeadMadPGMV}, where it is denoted
as $\mathcal{M}_{ult}[G]$, however, we'll gloss over that subtlety
here.} Note that since $(\mathbb{C}_{Ord})^{\mathcal{M}}$ is a pretame
forcing over $\mathcal{M}$, we see that $\mathcal{M}[G]$ satisfies
$ZFC^{-}$. A standard model $\mathcal{M}^{G}$ is then formed from
the set $\mathbb{W}^{\mathcal{M}^{G}}$ of generic refinements of
$\mathcal{M}[G\restriction\alpha]$ for each $\alpha\in Ord^{\mathcal{M}}$.\footnote{Or a little more precisely, $\mathcal{N}=\langle N,\in^{\mathcal{N}}\rangle$
is a world in $\mathcal{M}^{G}$ is there is some $\alpha\in Ord^{\mathcal{M}}$
and $r\in\mathcal{M}[G\restriction\alpha]$ such that $x\in N$ iff
$\langle\mathcal{M}[G],\mathcal{M},\in^{\mathcal{M}[G]}\rangle\models x\in(\mathbb{W}_{r})^{\mathcal{M}[G\restriction\alpha]}$. } Thus, $\mathcal{M}^{G}=\langle\mathbb{W}^{\mathcal{M}^{G}},\mathcal{M}[G],\in^{\mathcal{M}[G]}$.
We then recall that:
\begin{fact}
(Steel, Theorem 26 and 27 in \citep{MeadMadPGMV}) (1) If $\mathcal{M}$
is a model of $ZFC$ and $G$ is $(\mathbb{C}_{Ord})^{\mathcal{M}}$-generic
over $\mathcal{M}$, then $\mathcal{M}^{G}\models MV$.\label{fact:Steel,sound-comp}

(2) If $\mathcal{W}$ is a countable model of $MV$ and $\mathcal{M}$
is a world in $\mathcal{W}$, then for some $G$ that is $(\mathbb{C}_{Ord})^{\mathcal{M}}$-generic
over $\mathcal{M}$, we have\label{fact:Steel comp}
\[
\mathcal{W}=\mathcal{M}^{G}.
\]
\end{fact}

Informally speaking, given any model $\mathcal{M}$ of $ZFC$ and
a generic to collapse its ordinals, we may use them to obtain a standard
model of $MV$. And in the other direction, we see that if $\mathcal{W}$
is a \emph{countable }model of $MV$ then it must be a standard model.
We can use this to prove the theorem about the set part of $MV$. 
\begin{proof}
(of Theorem \ref{thm:MV|-ZFC_count}). Let $\mathcal{W}$ be a model
of $MV$. We'd like to apply Fact \ref{fact:Steel comp}(2), but for
that we need $\mathcal{W}$ to be countable. But if $\mathcal{W}$
is not countable, we can simply take a countable Skolem hull. Now
let $\mathcal{M}$ be an arbitrary world in $\mathcal{W}$. Fact \ref{fact:Steel comp}(2)
then tells us that we may fix some $(\mathbb{C}_{Ord})^{\mathcal{M}}$-generic
over $\mathcal{M}$ such that $\mathcal{W}=\mathcal{M}^{G}$. Then
since $(\mathbb{C}_{Ord})^{\mathcal{M}}$ is pretame, the set domain
$\mathcal{M}[G]$ of $\mathcal{M}^{G}$ satisfies $ZFC^{-}$. Moreover,
Extension ensures that every set is countable. 
\end{proof}
The Fact above is stated quite generally in that it accommodates
the worlds that are nonstandard in the sense that their membership
relation is not well-founded. But most (although not all) of the time,
we'll be dealing with models in which every world is transitive. In
those cases, we know that $\in^{\mathcal{W}}$ is just the \emph{real
}$\in$ restricted to $\mathbb{D}^{\mathcal{W}}$. Moreover, in these
case we can, without loss of generality, think of $\mathbb{W}^{\mathcal{W}}$
as literally being the set of those transitive sets. If we do this
then we see that $\mathbb{D}^{\mathcal{W}}=\bigcup\mathcal{W}$.\footnote{We'll see later that the converse does not hold: we cannot recover
the world domain from the set domain.} Thus, when we deal with transitive models of $MV$, we see that $\mathbb{D}^{\mathcal{W}}$
and $\in^{\mathcal{W}}$ are redundant and so in such cases, we'll
often abuse notation and let $\mathcal{W}$ just be its world domain
$\mathbb{W}^{\mathcal{W}}$. Now if we restrict our attention countable
transitive models we get a particularly nice model-theoretic representation
$MV$. 
\begin{cor}
For countable $\mathcal{L}_{MV}$-structures $\mathcal{W}$ containing
transitive worlds, the following are equivalent:\label{cor:Steel}
\begin{enumerate}
\item $\mathcal{W}\models MV$; and
\item $\mathcal{W}=M^{G}$ for some $G$ that is $Col(\omega,<Ord^{M})$-generic
over $M.$
\end{enumerate}
\end{cor}

For a helpful way to think about standard models of $MV$, it is helpful
to recall the proof behind Solovay's famous model where the reals
are Lebesgue measurable, have the property of Baire and the perfect
set property \citep{Solovay1970AMO}. We begin that proof by supposing
that there is an inaccessible cardinal $\kappa$ and then consider
a model where all the ordinals below $\kappa$ have been collapsed
using a generic $G$ for $\mathbb{C}_{\kappa}$. This was the first
of many similar models obtained by collapsing large cardinals. It
is a familiar playground for set theorists and -- as we'll see many
of the techniques for working there translate quite directly to models
of $MV$. The main difference is that to make a standard model of
MV we don't just collapse some of the ordinals, we collapse all of
them. As we've seen this gives us our set domain. We might then regard
the worlds of the multiverse are recording the history of all possible
forcings along the road to this great collapse.

Standard models arguably provide the most natural tool for understanding
$MV$ and how it is connected to $ZFC$. Moreover, this bridge provides
a helpful and well understood toolbox for proving theorems about $MV$.
However, not all models of $MV$ are standard. Given that we concerned
in this paper with questions around categoricity, it will be helpful
to survey the landscape beyond the realm of standard models. Of course,
$MV$ is just a two-sorted theory in first order logic and so, by
the upward L\"{o}wenheim-Skolem theorem, it is obvious that there are
models of $MV$ of every possible cardinality. But if we place some
natural constraints on the models and the worlds they contain a more
interesting picture emerges. For example, given Corollary \ref{cor:Steel}
we might wonder if every model $\mathcal{W}$ of $MV$ containing
countable transitive models must itself be countable. This is not
the case. 
\begin{prop}
If $MA(\kappa)$ holds for $\kappa\geq\omega$, then there are models
$\mathcal{W}$ of $MV$ with $|\mathcal{W}|=\kappa^{+}$ where every
world $M$ in $\mathcal{W}$ is countable and transitive (if there
is a countable transitive model of $ZFC$). Moreover, any such $\mathcal{W}$
will be nonstandard. 
\end{prop}

\begin{proof}
We start by observing that such models cannot be standard. To see
this, suppose $M$ is a world in $\mathcal{W}$ and $G$ is $\mathbb{C}_{Ord}^{M}$-generic
over $M$. Then for all $\alpha\in Ord^{M}<\omega_{1}$, $M[G\restriction\alpha]$
is countable and so it has countably many generic refinements, $\mathsf{W}_{r}^{M[G\restriction\alpha]}$,
indexed by elements $r\in M[G\restriction\alpha${]}. Thus, the worlds
of $M^{G}$ are a countable union of countable sets, while $\mathcal{W}$
is uncountable.

Now we prove that such models exist. Let $M$ be a countable transitive
model of $ZFC$. Before we build our target model, first observe that
if $G_{0}$ and $G_{1}$ are $\mathbb{C}_{Ord^{M}}$-generic over
$M$, then we can easily find $G_{2}$ that is $\mathbb{C}_{Ord^{M}}$-generic
over $M$ where $M[G_{2}]=M[G_{0}\times G_{1}]$. Thus, we obtain
a standard model $M^{G_{2}}$ extending both $M^{G_{0}}$ and $M^{G_{1}}$.
Clearly, this can be generalized to any finite set of $\mathbb{C}_{Ord^{M}}$-generics.
We might say that any finite collection standard models based on $M$
can be amalgamated into another standard model. Similar remarks apply
to countable collections of standard models, but as we've seen no
amalgamation fact can hold for uncountable collections of standard
models. We address this problem by inductively defining a sequence
$\langle G_{\alpha}\rangle_{\alpha<\kappa^{+}}$ of filters that are
mutually $\mathbb{C}_{Ord^{M}}$-generic over $M$. Let $\alpha<\kappa^{+}$
and suppose that we've already defined $\langle G_{\beta}\rangle_{\beta<\alpha}$
such that for all finite $X\subseteq\alpha$, $\prod_{\gamma\in X}G_{\gamma}$
is $\prod_{\gamma\in X}\mathbb{C}_{Ord^{M}}$-generic over $M$.\footnote{Note that $\prod_{\gamma\in X}\mathbb{C}_{Ord^{M}}$ is just $(\mathbb{C}_{Ord^{M}})^{|X|}$.}
We wish to define $G_{\alpha}$ so that $\langle G_{\beta}\rangle_{\beta<\alpha+1}$
has the analogous property at $\alpha+1$. For finite $X\subseteq\alpha$,
let 
\[
\mathcal{D}_{X}=\{D\in M[\prod_{\gamma\in X}G_{\gamma}]\ |\ D\text{ is dense in }\mathbb{C}_{Ord^{M}}\}.
\]
Then let $\mathcal{D}=\bigcup_{X\in[\alpha]^{<\omega}}\mathcal{D}_{X}$.
Since $\alpha<\kappa^{+}$, we see $|\mathcal{D}|\leq\kappa$ and
since $\mathbb{C}_{Ord^{M}}$ has the ccc (in $V$), we may use $MA(\kappa)$
we may fix a filter $G_{\alpha}$ where $G_{\alpha}\cap D\neq\emptyset$
for all $D\in\mathcal{D}$. This ensures that for all finite $X\subseteq\alpha$,
$G_{\alpha}$ is $\mathbb{C}_{Ord^{M}}$-generic over $M[\prod_{\gamma\in X}G_{\gamma}]$
and application of the product lemma completes our inductive argument.
Finally, we then let $\mathcal{W}$ be the set of generic refinements
of models of the form $M[\prod_{\gamma\in X}(G_{\gamma}\restriction\alpha)]$
where $X\in[\kappa^{+}]^{<\omega}$ and $\alpha\in Ord^{M}$. 
\end{proof}
Thus, we see that in contexts where some form of Martin's axiom holds,
there will be large nonstandard models of $MV$ in which every world
is countable and transitive. Moreover, since $MA(\omega)$ is a theorem
of $ZFC$, we see that such nonstandard models are unavoidable.
\begin{cor}
There are models $\mathcal{W}$ of $MV$ with $|\mathcal{W}|=\aleph_{1}$
where every world $\mathcal{M}$ in $\mathcal{W}$ is countable and
transitive. 
\end{cor}

Given that the delivery of generic sets for models of $ZFC$ is often
dependent on those models being countable, we might also wonder whether
standard models of $MV$ can only contain countable models. This is
not the case. 
\begin{prop}
If $0^{\#}$ exists, there there are standard models $\mathcal{W}$
of $MV$ that contain uncountable transitive worlds.\label{prop:Hjorth}
\end{prop}

\begin{proof}
Since $0^{\#}$ exists, we know that $L_{\omega_{1}}$ satisfies $ZFC$.
To deliver a standard model of $MV$ as desired, it will suffice to
show that there is some $G$ that is $\mathbb{C}_{\omega_{1}}$-generic
over $L_{\omega_{1}}$. Since $\mathbb{C}_{\omega_{1}}$ is not countable,
the usual Baire category argument will not work. We address this by
defining the required generic inductively.\footnote{A version of this argument can be found in the first claim within
the proof of Theorem 3.5 in \citep{HjorthU2}. I learned it from a
MathOverflow response by Farmer Schlutzenberg.}

First, we observe that for all $\alpha<\omega_{1}$, $\mathcal{P}(\mathbb{C}_{\alpha})^{L}$
is countable and so we may obtain a $\mathbb{C}_{\alpha}$-generic
filter over $L_{\omega_{1}}$ for any such $\alpha$. To set up the
induction, fix the club sequence $\langle\kappa_{\alpha}\rangle_{\alpha\leq\omega_{1}}$
of indiscernibles below $\omega_{1}$ witnessing that $0^{\#}$ exists
and where $\kappa_{\omega_{1}}=\omega_{1}$. We aim to deliver a sequence
$\langle G_{\alpha}\rangle_{\alpha\leq\omega_{1}}$ such that:
\begin{enumerate}
\item $G_{\alpha}$ is $\mathbb{C}_{\kappa_{\alpha}}$-generic over $L$;
and 
\item $G_{\alpha}\subseteq G_{\beta}$ for all $\alpha\leq\beta\leq\omega_{1}$. 
\end{enumerate}
For successor ordinals, $\alpha+1$, we just find $H$ that is $\mathbb{C}_{[\kappa_{\alpha},\kappa_{\alpha+1})}$-generic
over $L[G_{\alpha}]$ and then $G_{\alpha+1}$ can be what is essentially
the product of $G_{\alpha}$ and $H$. At limit ordinals $\beta$,
we have to be a little more careful. Note first that each $\kappa_{\beta}$
is inaccessible in $L$. Thus, $L$ thinks that $\mathbb{C}_{\kappa_{\beta}}$
has the $\kappa_{\beta}$-cc property and so any maximal antichain
$A$ in $\mathbb{C}_{\kappa_{\beta}}$ will be a subset of $\mathbb{C}_{\kappa_{\alpha}}$
for some $\kappa_{\alpha}<\kappa_{\beta}$.\footnote{See the first lemma on page 15 of \citep{Solovay1970AMO} for the
canonical proof of this fact.} Given this, we can let $G_{\beta}=\bigcup_{\alpha<\beta}G_{\alpha}$. 
\end{proof}
This result depends on a strong assumption beyond $ZFC$. Without
it, it is possible to find models $\mathcal{W}$ of $MV$ where every
world is transitive, but $\mathcal{W}$ is not standard.
\begin{prop}
If $L_{\omega_{1}}\models ZFC$, then it is consistent that there
is a model $\mathcal{W}$ of $MV$ where every world $M$ in $\mathcal{W}$
is uncountable and transitive, but $\mathcal{W}$ is not standard.
\end{prop}

\begin{proof}
Again, we are going to build a multiverse from $L_{\omega_{1}}$.
However, we know from Proposition \ref{prop:Hjorth} that this cannot
be done within a model where $0^{\#}$ exists. More precisely, we
need to find a model where there is a $\mathbb{C}_{\alpha}$-generic
set for every $\alpha<\omega_{1}$, but no $\mathbb{C}_{\omega_{1}}$-generic
set exists. To obtain such a model, we use Silver's collapse forcing.
For $X\subseteq Ord$, we let $\mathbb{S}_{X}$ be the poset whose
domain is formed from functions $p:s\times Y\to Ord$ where we have:
the \emph{first domain}, $s\in[\omega]^{<\omega}$; the \emph{second
domain},\emph{ }$Y\in[X]^{\leq\omega}$; and for all $n\in s$ and
$\alpha\in Y$, $p(n,\alpha)<\alpha$.\footnote{See Section 20 in Cummings' chapter in \citep{foreman2009handbook}
for more information on this and a more general definition.} Silver's collapse forcing can be obtained from L\'{e}vy's collapse by
weakening the constraint on the second domain so that it merely has
to be countable rather than finite. This gives $\mathbb{S}_{\omega_{1}}$
quite different properties from $\mathbb{C}_{\omega_{1}}$. In particular,
$\mathbb{S}_{\omega_{1}}$ has, so to speak, a certain amount of countable
completeness in its second domain. 

Now let $H$ be $\mathbb{S}_{\omega_{1}}$-generic over $L$.\footnote{Strictly speaking, there is no reason to think such an $H$ exists.
However following usual conventions, we can think of ourselves as
now working \emph{virtually} within the context of $\Vdash_{\mathbb{S}_{\omega_{1}}}$
as calculated in $L$. This is sufficient for our consistency claim.
Ordinarily, such metamathematical subtleties can be largely ignored.
But given that this paper is so concerned with generic sets and whether
they exist, it is worth taking the time to take a little care around
this.} Since $\mathcal{P}(\mathbb{C}_{\alpha})^{L}$ is countable for all
$\alpha<\omega_{1}$, we see that there is a $\mathbb{C}_{\alpha}$
-generic over $L$ for each such $\alpha$. It will thus suffice to
show that $L[H]$ contains no sets $G$ that are $\mathbb{C}_{\omega_{1}}$-generic
over $L$.\footnote{Our proof is an elaboration of a sketch by Mohammad Golshani on MathOverflow.
Nothing in the argument relies on us using $L$ rather than $V$.
It just lines up more smoothly with the story of the proof. I'm grateful
to Jason Chen for drawing my attention to it.} To demonstrate this, we prove a claim that suffices for the proposition.
\begin{claim*}
(1) If $A\in[\omega_{1}]^{\omega_{1}}\cap V[H]$, then there is a
countably infinite $X\subseteq A$ in $L$.

(2) If $G$ is $\mathbb{C}_{\omega_{1}}$-generic over $L$, then
for some $A\in[\omega_{1}]^{\omega_{1}}\cap L[G]$ there is no countably
infinite $X\subseteq A$ in $L$. 
\end{claim*}
To see that this is sufficient, suppose toward a contradiction that
$G\in V[H]$ is $\mathbb{C}_{\omega_{1}}$-generic over $L$. Then
by (1) we may fix $A\in[\kappa]^{\kappa}\cap V[G]$ which has no countably
infinite subset in $L$. But then we see that $A\in V[G]\subseteq V[H]$
and so by (2) $A$ must have a countable infinite subset, which is
impossible. 

To prove (1), we make an argument that exploits the version of countable
completeness possessed by the second domain of $\mathbb{S}_{\omega_{1}}$.\footnote{The argument is a simple generalization of the standard argument that
forcings closed under $\omega$-sequences don't add any new $\omega$-sequences.
See, for example, Theorem VII.6.14 in \citep{KunenST}.} We suppose toward a contradiction that $L$ contains no injection
from $\omega$ into $A$. We continue by letting $\dot{A}\in L^{\mathbb{S}_{\omega_{1}}}$
be a name such that $\dot{A}_{H}=A$ and fixing $p\in H$ that forces
the statement: there is is no injection $f:\omega\to\dot{A}$ with
$f\in\check{V}$. Working in $L$, we then define sequences $\langle p_{n}\rangle_{n\in\omega}$
and $\langle\alpha_{n}\rangle_{n\in\omega}$ such that $p_{0}\leq p$
and for all $n\in\omega$:
\begin{itemize}
\item $p_{n}\Vdash\alpha_{n}\in\dot{A}$;
\item $\alpha_{n+1}>\alpha_{n}$; 
\item $p_{n+1}\leq p_{n}$; and 
\item $dom(p_{n})=m\times Y$ where $m\in\omega$ is least among the second
domains of conditions satisfying the first three bullet points.
\end{itemize}
We claim that there is some $k\in\omega$ such that for all $n\in\omega$,
if $m$ is the second domain of $p_{n}$, then $m\leq k$. To see
this, we move back to $V$ and observe that since $|A|=\kappa$ and
there are only countably many possibilities for the first domains
of conditions, there must be some $k\in\omega$ such that there are
infinitely many $\alpha\in A$ with some $p\in H$ with $p\Vdash\alpha\in\dot{A}$
where the second domain of $p$ is $<\kappa$. Since we have a bound
on the first domain, it can be seen that $q=\bigcup_{n\in\omega}p_{n}\in\mathbb{S}_{\omega_{1}}$.
Finally, let $J$ be $\mathbb{S}_{\omega_{1}}$-generic over $L$
with $q\in J$. Then we see that $\langle\alpha_{n}\rangle_{n\in\omega}$
is an injection from $\omega$ into $A$ that is an element of $L$,
which gives a contradiction completing our proof of (1).

To prove (2), we provide a $\mathbb{C}_{\omega_{1}}$-name $\dot{A}$
for a set $A$ whose realization has no countably infinite subset
in $L$. More specifically, we let 
\[
\dot{A}=\{\langle\check{\alpha},p\rangle\in\check{\omega}_{1}\times\mathbb{C}_{\omega_{1}}\ |\ p(0,\alpha)=1\}
\]
and $A=\dot{A}_{G}$. The see the idea behind this name, let $g=\bigcup G:\omega\times\omega_{1}\to\omega_{1}$
and note that $\alpha\in\dot{A}_{G}$ iff $g(0,\alpha)=1$. So we
are restricting our attention to how conditions $p\in\mathbb{C}_{\omega_{1}}$
act on $0$ in the first domain and we only care about whether the
output is $1$ or not. This gives us a name for an element of $[\omega_{1}]^{\omega_{1}}\cap L[G]$.

Next we claim that $A$ is $Add(\omega_{1},1)$-generic over $V$.
To see this, we recall a basic, abstract forcing fact. Let us say
that $\sigma:\mathbb{P}\to\mathbb{Q}$ is \emph{weakly dense} if $\sigma$
is order preserving and for all $p\in\mathbb{P}$ and $q\in\mathbb{Q}$
with $q\leq\sigma(p)$, there is some $p^{*}\in\mathbb{P}$ such that
$p\not\perp p^{*}$ and $\sigma^{*}(p)\leq q$. It is well-known then
that if $J$ is $\mathbb{P}$-generic over $V$, then the upward closure
of $\sigma``J$ is $\mathbb{Q}$-generic over $V$.\footnote{See Lemma 15.45 in \citep{JechST} for a proof.}
Now let us construe $Add(\omega_{1},1)$ as the poset whose conditions
$q$ are finite partial functions from $\omega_{1}$ to $2$; and
let $\sigma:\mathbb{C}_{\omega_{1}}\to Add(\omega_{1},1)$ be such
that for all $p\in\mathbb{C}_{\omega_{1}}$, $\sigma(p)(\alpha)=1$
iff $p(0,\alpha)=1$.\footnote{This is not strictly correct. Since $p(0,0)$ has no possible output
and $p(0,1)$ can only output $0$, we can't just use $\alpha$ on
both sides. Fixing this makes a simple idea ugly on the page, so I'll
stick with the sloppiness. It is, however, easily addressed and we
leave that to the reader.} It can be seen that $\sigma$ is a weakly dense embedding and so
$J=\sigma``G$ is $Add(\omega_{1},1)$-generic over $L$. Moreover,
it is not difficult to see that $A=\bigcup J$. Finally, suppose toward
a contradiction that $X\in L$ and $X\subseteq A$. Now let $\dot{J}$
be the canonical name for a $Add(\omega_{1},1)$-generic set.\footnote{Here we mean that $\dot{J}=\{\langle\check{q},q\rangle\ |\ q\in\mathbb{Q}\}$.}
Then we may fix $q\in J$ such that $q\Vdash X\subseteq\bigcup\dot{J}$.
But this means that we have a finite condition $q\in Add(\omega_{1},1)$
determining infinitely many membership facts about $\bigcup\dot{J}$,
which is impossible.
\end{proof}
Taking a little stock, we've seen that while standard models of $MV$
are common, they are not ubiquitous. We should also note that in the
models of $MV$ above that contained uncountable transitive worlds,
$W$, the ordinals of those worlds were identical to $\omega_{1}$.
We might then wonder if we can have models of $MV$ containing worlds
with a longer sequence of ordinals. This is not the case.\footnote{This was observed by Usuba in the remarks following Definition 7.16
of \citep{UsubaDDGpub}.} To see this, we start by noting that the ``external'' cardinality
of any pair of sets in some $\mathcal{W}$ must be the same. 
\begin{prop}
If $\mathcal{W}$ is a model of $MV$ then for all $x,y$ in the set
domain of $\mathcal{W}$\label{prop:nonstandMVmod}
\[
|\{z\in M\ |\ z\in^{\mathcal{\mathcal{W}}}x\}|=|\{z\in M\ |\ z\in^{\mathcal{W}}y\}|.
\]
\end{prop}

\begin{proof}
First note that since $MV$ implies $ZFC_{count}^{-}$, if $x$ is
any element of the set domain of $\mathcal{W}$, then there is some
$f_{x}$, also in $\mathcal{W}$'s set domain, such that $\mathcal{W}\models f_{x}:\omega\cong x$.
It is then clear that $f$ can be pulled back into $V$ to obtain
a bijection 
\[
f_{x}^{*}:\omega^{\mathcal{W}}\cong\{z\in M\ |\ z\in^{\mathcal{\mathcal{W}}}x\}.
\]
Thus, $f_{y}^{*}\circ f_{x}^{*-1}$ witnesses the required claim. 
\end{proof}
The nonstandardness of the worlds of such a model then follows easily.
\begin{cor}
If $\mathcal{W}$ is a model of $MV$ in which there is a world $\mathcal{M}$
and some $x\in\mathcal{M}$ such that
\[
\{y\in M\ |\ y\in^{\mathcal{W}}x\}
\]
is uncountable, then $\mathcal{W}$ is not an $\omega$-model.
\end{cor}

\begin{proof}
Let $x$ be a member of the set domain of $\mathcal{W}$ be such that
its $\in^{\mathcal{W}}$-members are uncountable. Then using the proof
of Proposition \ref{prop:nonstandMVmod}, we may fix a bijection $f_{x}^{*}$
between $\omega^{\mathcal{W}}$ and the $\in^{\mathcal{W}}$-members
of $x$. This mean $\omega^{\mathcal{W}}$ is uncountable and thus
$\mathcal{W}$ cannot be an $\omega$-model. 
\end{proof}

\section{When $MV$ is not internally categorical\label{sec:When--is}}

In this section, we show $MV$ in its standard axiomatization is
not internally categorical. The idea behind the proof is very simple.
We are going to take a standard model $\mathcal{W}$ of $MV$ and
some world $U$ from $\mathcal{W}$. We are then going to take a class
generic extension $U[G]$ of $\mathcal{W}$ that cannot be accessed
from $U$ by set forcing but which is also a world in a multiverse
with the same set domain as $\mathcal{W}$.
\begin{thm}
There exist distinct models $\mathcal{W}_{0},\mathcal{W}_{1}$ of
$MV$ whose set domains are identical, assuming there is a countable
transitive model of $ZFC$.\label{thm:There-exist-models}
\end{thm}

\begin{proof}
Our goal is to describe a pair of models of $MV$ whose set domains
are identical. Let us start with a countable transitive model $M$
of $ZFC$. Let $\mathbb{Q}\in M$ be $M$'s version of the Easton
support product of forcings that add a Cohen subset to every regular
cardinal.\footnote{This is also the forcing used by Reitz to show that there is a model
of $ZFC$ that doesn't satisfy the Bedrock axiom. See Theorem 5.3
in \citep{ReitzTGA}.} Now let $H$ be $\mathbb{Q}$-generic over $M$ and observe that
this is a class generic extension of $M$ for which there can be no
set generic extension $M[g]$ of $M$ with $M[g]=M[H]$. Next, we
let $J$ be $\mathbb{C}_{Ord^{M}}$-generic over $M[H]$. Thus, we
have 
\[
M\subsetneq M[H]\subsetneq M[H\times J].
\]
We plan to show that there is some $G$ that is $\mathbb{C}_{Ord^{M}}$-generic
over $M$ with $M[G]=M[H\times J]$. To see that this will suffice,
note that this will give us two standard models $M^{G}$ and $M[H]^{J}$
of $MV$ with the same set domain. However, their world domains cannot
be the same since $M[H]$ is not a world in $M^{G}$.

To prove our target claim, we shall obtain the required set $G$ from
a generic filter on a class poset $\mathbb{P}^{*}$ over $M[H\times J]$
that will assemble $G$ through a sequence of approximations. More
specifically, we let $\mathbb{P}^{*}$ consist of sets $g\in M[H\times J]$
that are $\mathbb{C}_{\alpha}$-generic over $M$ for some $\alpha\in Ord^{M}$,
and we order $\mathbb{P}^{*}$ by reverse inclusion.\footnote{A similar forcing is used in the proof of Theorem 3.1.5 in \citep{larsonstationary}.
Indeed, it can be used to give a more elegant proof of Theorem 27
in \citep{MeadMadPGMV}. I learned this from Gabe Goldberg.} Now let $G^{*}$ be $\mathbb{P}^{*}$-generic over $M[H\times J]$
and let $G=\bigcup G^{*}$. First, we claim that $G$ is $\mathbb{C}_{Ord^{M}}$-generic
over $M$. To see this let $A\subseteq\mathbb{C}_{Ord^{M}}$ be a
maximal anti-chain that is definable in $M$. It can be seen that
$\mathbb{C}_{Ord^{M}}$ is $Ord$-cc in $M$ and so there must be
some $\alpha\in Ord^{M}$ such that $A\subseteq\mathbb{C}_{\alpha}$.\footnote{This can be established with a trivial generalization of the proof
of first lemma on page 15 of \citep{Solovay1970AMO}.} Thus, it will suffice to show that there is some $g\in G^{*}$ where
$g\cap\mathbb{C}_{\alpha}\ne\emptyset$. To see this, we exploit the
genericity of $G^{*}$ and make a density argument. In particular,
we let 
\[
D=\{g\in\mathbb{P}^{*}\ |\ g\text{ is }\mathbb{C}_{\alpha}\text{-generic}/M\}
\]
and show that $D$ is pre-dense in $\mathbb{P}^{*}$. To see this,
let $h\in\mathbb{P}^{*}$ and suppose $h$ is $\mathbb{C}_{\beta}$-generic
over $M$ where $\beta\in Ord^{M}$. If $\beta\geq\alpha$, we would
be done, so suppose $\beta<\alpha$. It will suffice to show that
for some regular cardinal $\kappa>\alpha$ there is some $h^{*}\in M[H\times J]$
that is $\mathbb{C}_{[\beta,\kappa+1)}$-generic over $M[h]$. Let
us start by fixing a $(\mathbb{Q}\times\mathbb{C}_{Ord^{M}})$-name,
$\dot{h}$, such that $\dot{h}_{H\times J}=h$. Now since $\dot{h}$
is just a set it is clear that its transitive closure can only contain
a set of conditions from $\mathbb{C}\times\mathbb{C}_{Ord^{M}}$.
Thus, we may fix some $\gamma\in Ord^{M}$ such that $h=\dot{h}_{(H\times J)\restriction\gamma}$
and so $h\in M[(H\times J)\restriction\gamma]$. Now let $\kappa$
be a regular $M$-cardinal $>\gamma,\alpha$. It is clear that $J_{\{\kappa\}}$
is $\mathbb{C}_{\{\kappa\}}$-generic over $M[(H\times J)\restriction\gamma]$.\footnote{We write $H\restriction\alpha$ to denote those $p\in H$ with $dom(p)\subseteq\alpha$.}
Moreover, it can be seen that $\mathbb{C}_{\{\kappa\}}$ is forcing
equivalent\footnote{See Proposition 10.20 and and the proof of Proposition 10.21 in \citep{kanamori2003higher}
for more details.} to $\mathbb{C}_{[\beta,\kappa+1)}$ and so we may fix $h^{*}$ that
is $\mathbb{C}_{[\beta,\kappa+1)}$-generic over $M[(H\times J)\restriction\gamma]\supseteq M[h]$
as required.

Finally, we claim that $M[G]=M[H\times J]$. Clearly, $M[G]\subseteq M[H\times J]$
so we want to show that every $x\in M[H\times J]$ is in $M[G]$.
To do this, we again exploit the genericity of $G^{*}$ by showing
that if $x\in M[H\times J]$ then 
\[
E=\{g\in\mathbb{P}^{*}\ |\ x\in M[g]\}
\]
is dense in $\mathbb{P}^{*}$. To see this, let $\dot{x}$ be a $(\mathbb{Q}\times\mathbb{C}_{Ord^{M}})$-name
in $M$ such that $\dot{x}_{(H\times J)}=x$. Since $\dot{x}$ is
just a set, we may fix, as in the previous claim, some $\alpha\in Ord^{M}$
such that $\dot{x}_{(H\times J)\restriction\alpha}=x$. Moreover,
we may assume without loss of generality that $\alpha=\kappa+1$ for
some regular $M$-cardinal. Then it can be seen that $\mathbb{C}_{\{\kappa\}}$,
$\mathbb{C}_{\kappa+1}$ and $(\mathbb{Q}\times\mathbb{C}_{Ord^{M}})\restriction(\kappa+1)$
are all forcing equivalent to each other. Thus, there is some $g$
that is $\mathbb{C}_{\kappa+1}$-generic over $M$ with $x\in M[g]$.
\end{proof}
We shall now use this fact to establish that $MV$ is not internally
categorical. However, the underlying idea is quite obvious. The models
given in the theorem above are not isomorphic, even though they share
a set domain. Before we state the claim, let us first described our
target theory. Let $\mathcal{L}_{MV}(\in_{0},\in_{1})$ be the language
with a set sort and a world sort, but which has two membership relations
that behave as in $\mathcal{L}_{MV}$. Let $MV(\in_{0},\in_{1})$
be the theory consisting of $\in_{0}$ and $\in_{1}$ versions of
every axiom of $MV$ where we allow formulae involving both $\in_{0}$
and $\in_{1}$ into both $\in_{0}$ and $\in_{1}$ versions of all
schemata.
\begin{cor}
$MV(\in_{0},\in_{1})$ is not internally categorical, if there is
a countable transitive model of $ZFC$. 
\end{cor}

\begin{proof}
Using the proof of Theorem \ref{thm:There-exist-models}, let us start
with a countable transitive model $M$ of $ZFC$, but this time suppose
that $M$ satisfies $V=L$. We may then obtain standard models $\mathcal{W}_{0}\neq\mathcal{W}_{1}$
of $MV$ with the same world domain whose worlds are countable transitive
models satisfying $ZFC$. In particular, $\mathcal{W}_{1}$ is generated
from a world $N$ that is obtained from $M$ by adding a Cohen subset
to every $M$-regular cardinal. It can be seen that no world in $\mathcal{W}_{1}$
satisfies the Ground Axiom.\footnote{See Theorem 5.3 in \citep{ReitzTGA} for a proof.}

We now define a model $\langle\mathbb{W},\mathbb{D},\in_{0},\in_{1}\rangle$
of $MV(\in_{0},\in_{1})$. We let the set domain $\mathbb{D}=\bigcup\mathcal{W}_{0}=\mathcal{W}_{1}$.
It is clear that $\mathcal{W}_{0}$ and $\mathcal{W}_{1}$ are both
countable, so we may fix a bijection $\sigma:\mathcal{W}_{1}\to\mathcal{W}_{0}$.
Let $\mathbb{W}=\mathcal{W}_{0}$. We then let $\in_{0}$ be $\in$
restricted to $\bigcup\mathcal{W}_{0}$. And for $\in_{1}$, we let:
$x\in_{1}y$ iff $x\in y$ where $x$ and $y$ are both sets; and
we let $x\in_{1}M$ iff $x\in\sigma(M)$ when $x$ is a set and $M$
is a world. Since $\in_{0}$ and $\in_{1}$ are the same on the set
domain, it is easy to see that having both $\in_{0}$ and $\in_{1}$
available to the (relativized) Collection and Separation schemas,
used in $MV$, makes no difference. Thus, $\langle\mathbb{W},\mathbb{D},\in_{0},\in_{1}\rangle$
satisfies $MV(\in_{0},\in_{1})$. Finally, to see that there is no
definable isomorphism between $\mathcal{W}_{0}=\langle\mathbb{W},\mathbb{D},\in_{0}\rangle$
and $\mathcal{W}_{1}=\langle\mathbb{W},\mathbb{D},\in_{1}\rangle$
it suffices to observe that they are not even elementary equivalent.
To see this we note that some world in $\mathcal{W}_{0}$ satisfies
the Ground Axiom since $\mathcal{W}_{0}$ is generated from a model
$M$ satisfying $V=L$. But we've already seen that no world in $\mathcal{W}_{1}$
satisfies the Ground Axiom. 
\end{proof}
Thus, we see that $MV$ is not internally categorical, but it is natural
to wonder what would happen if we augmented $MV$ with axioms demanding
that certain large cardinals exist at some, or perhaps, all worlds.\footnote{See Section 5 of \citep{SteelGP} for some discussion of this.}
This can be quite attractive since, if every world contains a Woodin
cardinal, then it can be seen that every world must contain a proper
class of them. Moreover, if there is a proper class of Woodin cardinals
in every world then the theory of $L(\mathbb{R})$ and analysis will
be the same in every world.\footnote{See Chapter 3 of \citep{larsonstationary} for a proof of this.}
Beyond the issue of whether it is interesting to add large cardinals
to $MV$, it is also known that large cardinals can exert a powerful
influence on the structure of the models of $MV$. In particular,
\citet{UsubaDDGpub} has shown that if a world has an extendible cardinal,
then the multiverse has a core, or bedrock; i.e., there is a world
that is a submodel of every other world and every other world is a
generic extension of it. This could give us reason to doubt that the
failure internal categoricity will be preserved into strengthenings
of $MV$. These doubts appear to be ill-founded although some work
is required to show this. 

When we argued that $MV$ is not internally categorical, it sufficed
to prove that we could have distinct models of $MV$ with the same
world domain. However, the proof of this relied on a class forcing
argument. We started with a model of $ZFC$ and added a class of Cohen
sets. This forcing trivially preserves $ZFC$, but some remarks are
warranted with regard to the preservation of large cardinal assumptions.
This is not a small question, but rather overlaps an entire subfield
of set theory concerned with the preservation and indestructibility
of large cardinal axioms through forcing.\footnote{See Cummings' chapter in \citep{foreman2009handbook} for an introduction
to this topic.} As such, we shall content ourselves with small sample of large cardinal
axioms that are intended to cut across a wide swathe of consistency
strengths. 
\begin{lem}
Consider the theories obtained by expanding $MV$ with each of the
following axioms:
\begin{enumerate}
\item There is a world with an inaccessible cardinal;
\item There is a world with a measurable cardinal; 
\item There is a world with a supercompact cardinal; and
\item There is a world with an extendible cardinal.
\end{enumerate}
Assuming sufficient large cardinals, each of these theories has a
pair of distinct models $\mathcal{W}_{0}\neq\mathcal{W}_{1}$whose
set domains are identical.\label{lem:Assuming-sufficient-large}
\end{lem}

For the last two of these theories, we make use of the following facts
about supercompact and extendible cardinals.
\begin{fact}
\citep{ReitzTGA} If $\kappa$ is a supercompact cardinal, then\footnote{This can be proved using two moves from \citep{ReitzTGA}. Starting
in $V$ with a supercompact cardinal, $\kappa$, we use the proof
of Theorem 3.9 to class force to a model $V[G]$ where the continuum
coding axiom holds and $\kappa$ remains supercompact and is indestructible
by $<\kappa$-directed forcing. Then we the proof of Theorem 5.5 to
add a Cohen subset to every regular $\lambda$ such that $2^{<\lambda}=\lambda$.
It can then be seen via the indestructibility that the supercompactness
of $\kappa$ is preserved by this extension.} there is a class forcing extension $V[G]$ in which $\kappa$ is
supercompact and if we class force over $V[G]$ to add a Cohen subset
to every regular cardinal $\lambda$ in $V[G]$ where $2^{<\lambda}=\lambda$,
then $\kappa$ remains supercompact in that extension.\label{fact:ReitzSC}
\end{fact}

\begin{fact}
\citep{BagariaExtPres} Suppose Vopenka's principle holds. Then there
is an extendible cardinal $\kappa$. Moreover,\footnote{This is essentially a suboptimal version Corollary 5.13 in \citep{BagariaExtPres}.
The assumptions of the Corollary follow from Theorems 6.1-3 as well
as Theorem 2.4 and Corollary 2.6 in \citep{BagariaExtPres}.} if $G$ is generic for the Easton support forcing adding a Cohen
subset to every regular cardinal, then $\kappa$ remains extendible
in $V[G]$.\label{fact:Bagaria}
\end{fact}

\begin{proof}
(1) The argument used in Theorem \ref{thm:There-exist-models} suffices
since the class forcing adding Cohen subsets to each regular cardinal
preserves inaccessible cardinals. 

(2) The argument from Theorem \ref{thm:There-exist-models} can be
adapted by only adding Cohen subsets of regular cardinals greater
than the least measurable cardinal of our initial model $M$. 

(3) Suppose that $\kappa<\lambda$ are supercompact and inaccessible
respectively. Then take the collapse of a countable Skolem hull to
obtain a countable transitive model $M$ that things some $\bar{\kappa}$
is supercompact. Now using Fact \ref{fact:ReitzSC}, we may adapt
the argument of Theorem \ref{thm:There-exist-models} to deal with
a forcing that only adds Cohen sets to regular cardinals $\delta$
such that $2^{<\delta}=\delta$.

(4) Suppose that $\lambda$ is a Vopenka cardinal. Then take the collapse
of a countable Skolem hull of $V_{\lambda}$ to obtain a countable
transitive model $M$ satisfying Vopenka's principle. Let $\kappa\in M$
be such that $M$ thinks $\kappa$ is extendible. Fact \ref{fact:Bagaria},
then ensures that the class forcing used in Theorem \ref{thm:There-exist-models}
preserves the extendibility of $\kappa$. The rest of that argument
remains the same.\footnote{This argument will also work for the theory where every world has
an extendible cardinal and thus every world has a proper class of
extendible cardinals.}
\end{proof}
\begin{cor}
None of the theories considered in Lemma \ref{lem:Assuming-sufficient-large}
are internally categorical. 
\end{cor}

Thus, in its standard axiomatization we have seen that $MV$ and some
its large cardinal variants are not internally categorical. Intuitively,
this probably doesn't come as a surprise, however, there is value
in thinking about how we proved this. In a nutshell, the trick was
to find a pair of countable transitive models of $ZFC$ that share
a $\mathbb{C}_{Ord}$-collapse but where one is a (proper) class forcing
extension of the other. Reflecting on this issue motivates a modified
axiomatization of $MV$ that avoids this problem. 

\section{When $MV$ is internally categorical\label{sec:When--is-1}}

We are not reading to prove the more surprising result. There is an
alternative axiomatization $MV^{*}$ of $MV$ that \emph{is }internally
categorical. Let us begin by introducing the new axiom schemata. 
\begin{lyxlist}{00.00.0000}
\item [{(Global-$ZFC_{count}^{-}$)}] $\varphi$ for every axiom of $ZFC_{count}^{-}$. 
\item [{(World-domination)}] If $N$ is a definable inner model then there
is a world $W$, such that $N\subseteq W$.
\end{lyxlist}
We already saw above that Global-$ZFC_{count}^{-}$ follows from $MV$.
So -- at least from a mathematical perspective -- there is no harm
in adding it to $MV$. World-domination is more complicated. Informally,
it tells us that if you can define an inner model $N$ of $ZFC$ using
$\mathcal{L}_{MV}$, then there will be a world $W$ that dominates
in the sense that $W\supseteq N$. Let us call the theory obtained
by augmenting $MV$ with the schemata above $MV^{*}$. We shall show
first that $MV^{*}$ is internally categorical and then we will show
that $MV^{*}$ is an alternative axiomatization of $MV$ by demonstrating
that $MV^{*}$ follows from $MV$. 

We shall begin by establishing that $ZFC_{count}^{-}$ is internally
categorical. This may be of independent interest to those with countabilist
sympathies. However, it will also be helpful for our main proof since
we have the Global-$ZFC_{count}^{-}$ schema. We let $ZFC_{count}^{-}(\in_{0},\in_{1})$
be the theory articulated in $\mathcal{L}(\in_{0},\in_{1})$ that
is obtained by adding $\in_{0}$ and $\in_{1}$ versions of every
axiom of $ZFC_{count}^{-}$ and by allowing formulae involving both
$\in_{0}$ and $\in_{1}$ into both $\in_{0}$ and $\in_{1}$ versions
of the Separation and Collection schemata.
\begin{thm}
$ZFC_{count}^{-}(\in_{0},\in_{1})$ is internally categorical; i.e.,
we can prove in $ZFC_{count}^{-}(\in_{0},\in_{1})$ that there is
a definable isomorphism between $\in_{0}$ and $\in_{1}$.\label{thm:ZFCcount}
\end{thm}

\begin{proof}
We work informally in $ZFC_{count}^{-}(\in_{0},\in_{1})$. Roughly
speaking, our plan is to mutually collapse $\in_{0}$ and $\in_{1}$
within each other, then repeat the collapses and observe that there
is no room in between for difference. Let us write $\mathcal{V}_{0}$
for $\langle V,\in_{0}\rangle$ and $\mathcal{V}_{1}$ to denote $\langle V,\in_{1}\rangle$
where $V$ is the underlying universe.\footnote{Strictly speaking, I cannot refer to order $\langle V,\in_{0}\rangle$
within the framework of $ZFC_{count}^{-}(\in_{0},\in_{1})$ since
it is an ordered pair of proper classes. It is easy enough to eliminate
this abuse of notation, but it makes things less pleasant on the page,
so we shall continue to talk in this way. } Note that using the Set Induction schema, it is easy to see that
$\mathcal{V}_{0}$ thinks $\in_{1}$ is well-founded and that $\mathcal{V}_{1}$
thinks that $\in_{0}$ is well-founded. Given this, we may fix collapse
maps:
\begin{itemize}
\item $\pi_{1}:\mathcal{V}_{1}\cong\mathcal{M}_{1}=\langle M_{1},\in_{0}\rangle$;
and
\item $\pi_{0}:\mathcal{V}_{0}\cong\mathcal{M}_{0}=\langle M_{0},\in_{1}\rangle$.
\end{itemize}
Thus, $\mathcal{M}_{1}$ is the result of collapsing $\mathcal{V}_{1}$
in $\mathcal{V}_{0}$ and $\mathcal{M}_{0}$ is the result of collapsing
$\mathcal{V}_{0}$ in $\mathcal{V}_{1}$. Since these are isomorphisms,
we can then, so to speak, find a copy of $\mathcal{V}_{1}$ inside
$\mathcal{M}_{0}$. More formally, we have:
\begin{itemize}
\item $\pi_{1}\circ\pi_{0}``\mathcal{V}_{0}=\pi_{1}``\mathcal{M}_{0}\subseteq\mathcal{M}_{1}\subseteq\mathcal{V}_{0}$;
and
\item $\pi_{0}\circ\pi_{1}``\mathcal{V}_{1}=\pi_{0}``\mathcal{M}_{1}\subseteq\mathcal{M}_{0}\subseteq\mathcal{V}_{1}$.
\end{itemize}
Now if things worked out particularly nicely, we'd see that $\pi_{1}\circ\pi_{0}``\mathcal{V}_{0}$
was the collapse of $\mathcal{V}_{0}$ and thus was $\mathcal{V}_{0}$
itself. This would essentially complete the argument. However, for
all we know one of $\mathcal{V}_{0}$ and $\mathcal{V}_{1}$ is, so
to speak, shorter than the other and this means that the collapse
might not reach every element of the other model. We get around this
by noting that, according to both $\mathcal{V}_{0}$ and $\mathcal{V}_{1}$,
since every set is countable, every set can be coded by a set of natural
numbers. It will, thus suffice to show that the $\pi_{0}\circ\pi_{1}``\mathcal{V}_{0}$
and $\mathcal{V}_{0}$ have the same powerset of the naturals.

To show this, we start by reducing our notational burden by letting
$\mathcal{N}_{0}=\pi_{0}\circ\pi_{1}``\mathcal{V}_{0}$. It will then
suffice show that $\mathcal{P}(\omega)^{\mathcal{V}_{0}}\subseteq\mathcal{N}_{0}$
since the other direction is trivial. Let's start with some $x\in\mathcal{P}(\omega)^{\mathcal{V}_{0}}$.\footnote{Note that it is not strictly correct to speak of $\mathcal{P}(\omega)^{\mathcal{V}_{0}}$
since $\mathcal{V}_{0}$ contains no set of all subsets of its $\omega$.
However, we can instead regard $\mathcal{P}(\omega)^{\mathcal{V}_{0}}$
as the definable class of subsets of $\mathcal{V}_{0}$'s $\omega$. } It should be clear that $\langle trcl(\{x\})^{\mathcal{V}_{0}},\in\rangle$
is set-like according to $\mathcal{V}_{1}$ and so $\pi_{0}(x)$ is
defined. From here it is can be seen that $\pi_{1}\circ\pi_{0}(x)=x\in\mathcal{N}_{0}$
as required. To close things out, note that -- after unpacking the
coding -- this implies that $\mathcal{N}_{0}=\mathcal{V}_{0}$ and
so 
\[
\mathcal{V}_{0}\subseteq\mathcal{M}_{1}\subseteq\mathcal{V}_{0}
\]
which means that $\mathcal{V}_{0}=\mathcal{M}_{1}$ and thus, $\pi_{1}:\mathcal{V}_{1}\cong\mathcal{V}_{0}$
as required.
\end{proof}

\subsection{$MV^{*}$ is internally categorical}

We are now ready for the main result. Let $MV^{*}(\in_{0},\in_{1})$
be the theory articulated in $\mathcal{L}_{MV}(\in_{0},\in_{1})$
that consists of $\in_{0}$ and $\in_{1}$ versions of every axiom
of $MV^{*}$ where we allow formulae involving both $\in_{0}$ and
$\in_{1}$ into both $\in_{0}$ and $\in_{1}$ versions of all schemata.
\begin{thm}
$MV^{*}(\in_{0},\in_{1})$ is internally categorical; i.e., we can
prove in $MV^{*}(\in_{0},\in_{1})$ that there is a definable bijection
on the set domain $\sigma$ witnessing an isomorphism between $\in_{0}$
and $\in_{1}$ where $\sigma$ also determines an isomorphism on the
world domain.\label{thm:MV*intCat}
\end{thm}

\begin{proof}
We work in $MV^{*}(\in_{0},\in_{1})$. Let $\mathcal{V}_{0}=\langle V,\in_{0}\rangle$
and $\mathcal{V}_{1}=\langle V,\in_{1}\rangle$. By Fact \ref{thm:MV|-ZFC_count},
we see that both $\mathcal{V}_{0}$ and $\mathcal{V}_{1}$ satisfy
$ZFC_{count}^{-}$. Thus, we may use Theorem \ref{thm:ZFCcount},
to obtain a definable isomorphism $\sigma:\mathcal{V}_{0}\cong\mathcal{V}_{1}$
as we require. This leaves the world domain. For worlds $\in_{0}$-worlds,
$W$, we let $\sigma_{+}(W)=\sigma``W$. It will suffice to show that
$\sigma_{+}$ is a bijection. There are two things that might go wrong:
(1) $\sigma_{+}(W)$ might not be a world with respect to $\in_{1}$;
and (2) $\sigma_{+}$ might not be a surjection. 

Regarding (1), we see using ($*$) for $\in_{1}$ that there an
$\in_{1}$-world $U$ such that $\sigma_{+}(W)\subseteq U$. Moreover,
using ($*$) for $\in_{0}$, we also see that there is an $\in_{0}$-world
$W^{\dagger}$ such that $U\subseteq\sigma_{+}(W^{\dagger})$. By
the intermediate submodel theorem, in $\sigma_{+}(W^{\dagger})$,
we see that $\sigma_{+}(W)\subseteq U\subseteq\sigma_{+}(W^{\dagger})$
is a sequence of generic extensions. This means $\sigma_{+}(W)$ is
a generic refinement of $U$ and so $\sigma_{+}(W)$ is an $\in_{1}$-world
as we required. 

Regarding (2), we use a similar argument. Given an arbitrary $\in_{1}$-world
$U$, it will suffices to show that $\sigma_{+}^{-1}(U)$ is an $\in_{0}$-world.
To see this, we use ($*$) for $\in_{0}$, to fix an $\in_{0}$-world
$W$ such that $\sigma_{+}^{-1}(U)\subseteq W$ and so $U\subseteq\sigma_{+}(W)$.
Then using ($*$) for $\in_{1}$, we may fix an $\in_{1}$-world $U^{\dagger}$
such that $\sigma_{+}(W)\subseteq U^{\dagger}$. As in the previous
case, this implies that $\sigma_{+}(W)$ is a generic extension of
$U$ and thus, that $W$ is a generic extension of $\sigma_{+}^{-1}(U)$
and so $\sigma_{+}^{-1}(U)$ is an $\in_{0}$-world as we require.
\end{proof}

\subsection{$MV$ implies $MV^{*}$}

To complete our demonstration that an axiomatization of $MV$ is internally
categorical, we need to prove that $MV$ implies $MV^{*}$. To do
this, we need to show that:
\begin{enumerate}
\item World-domination is expressible in $\mathcal{L}_{MV}$; and
\item World-domination follows from $MV$.
\end{enumerate}

\subsubsection{World-domination is expressible}

To motivate why there might even be a problem here, let us recall
that the naive statement of the ordinary axioms of $MV$ is not obviously
expressible in $\mathcal{L}_{MV}$. For example, the Refinement axiom
is often stated as follows: ``If $U$ is a world, and $U=W[G]$,
where $G$ is $\mathbb{P}$-generic over $W$, then $W$ is a world''
\citep{SteelGP}. While this conveys the idea very clearly, it's not
so clear where $W$ came from. We appear to be quantifying over it
like a world before we even know that it is a world. The problem is
resolved by using Fact \ref{fact:Laver,-Woodin}, which allows for
the ``official'' formulation given in Section \ref{subsec:The--theory}.
As stated above, our formulation of World-domination has similar problems.
It appears to quantify over inner models, but in fact, it is more
correctly and much more formally stated as a schema as follows:
\begin{lyxlist}{00.00.0000}
\item [{(World-domination)}] For all sets $y_{0},...,y_{n}$ and for all
worlds $U_{0},...,U_{m}$, If 
\[
\varphi(x,y_{0},...,y_{n},U_{0},...,U_{m})
\]
 is a formula of $\mathcal{L}_{MV}$ that defines an \emph{inner}
\emph{model} of $ZFC$, then there is some world $W$ such that $\forall x(\varphi(x)\to x\in W)$.
\end{lyxlist}
By referring to formulae we are able to remove any quantification
over inner models. However, we still need to know when a particular
formula defines an inner model of $ZFC$. This can be addressed with
a classic fact. 
\begin{fact}
($ZFC^{-})$ A transitive class $M$ is an inner model of $ZF$ iff
it is closed under the G\"{o}del operations and is almost universal; i.e.,
for every subset $X\subseteq M$ there is some $Y\in M$ with $Y\supseteq X$.\footnote{For a proof see Theorem 13.9 in \citep{JechST}. The proof there is
conducted in $ZFC$ rather than $ZFC^{-}$, however, the proof is
the same. The only (cosmetic) modification occurs in establishing
almost universality of an inner model, $M$ of $ZFC$. We start with
$X\subseteq M$, and then obtain that $X\subseteq V_{\alpha}^{M}\in M$
rather than $X\subseteq V_{\alpha}\cap M$ since, in general, we cannot
guarantee that $V_{\alpha}$ exists in $ZFC^{-}$. }
\end{fact}

It is easy to see that being almost universal and closed under the
G\"{o}del operations is expressible in $\mathcal{L}_{\in}$. Then using
this and the fact that $MV$ implies $ZFC^{-}$, we can reformulate
World-domination by demanding that the class of $x$ such that $\varphi(x,y_{0},...,y_{n},U_{0},...,U_{m})$
is closed under the G\"{o}del operations and is almost universal. This
is easily expressed in $\mathcal{L}_{MV}$ and so World-domination
is expressible as a schema in $\mathcal{L}_{MV}$.

\subsubsection{World-domination follows}

Now we move onto our final task by showing that World-domination follows
from $MV$. Our strategy is to exploit the fact, explored earlier,
that models of $MV$ are closely related to models like Solovay's
where an inaccessible cardinals is collapsed to be $\omega_{1}$.
The key idea of this proof was provided by Gabe Goldberg and I'm grateful
to him for letting me use it here. We work our way gently toward our
target by starting with a warmup lemma, which contains the main idea
of the proof. We then use this and some lemmas from Steel and Goldberg
to push the result into less familiar setting of the generic multiverse,
which then gives us that result we require. 
\begin{lem}
Suppose $G$ is $\mathbb{C}_{Ord}$-generic over $V$ and that $M\subseteq V[G]$
is an inner model of $ZFC$ that is definable in $V[G]$ using $V$
as a parameter and parameters from $V$. Then $M\subseteq V$.\label{lem:warmup}
\end{lem}

\begin{proof}
It will suffice to show that every set of ordinals in $M$ is also
in $V$. To show this, we let $\lambda$ be an arbitrary ordinal and
let $Z=\mathcal{P}(\lambda)^{M}$. Since $\lambda$ is arbitrary,
it will suffice to show that $Z\subseteq V$. Our plan is to exploit
the homogeneity of $\mathbb{C}_{Ord}$ and chase a contradiction regarding
the cardinality of $Z$. The proof closely follows the argument used
by Solovay to establish that perfect set property in his well-known
model. Suppose toward a contradiction that $Z\nsubseteq V$ and fix
some $y\in V\backslash Z$ witnessing this.\footnote{A similar argument can be found in the proof of Theorem 1.1 in \citep{KanovieODrealsinGround}.}
By standard facts about collapse forcing, we may fix $\alpha<\kappa$
sufficiently large such that:
\begin{enumerate}
\item $y\in Z\in V[G\restriction\alpha]$; and
\item $Z$ is countable in $V[G\restriction\alpha]$.
\end{enumerate}
Using the homogeneity and absorption properties of $\mathbb{C}_{Ord}$
we may fix a formula $\varphi_{Z}(x)$ using $\lambda$ and elements
of $V$ as parameters, which is such that for any $H$ that is $\mathbb{C}_{\alpha}$-generic
over $V$ 
\[
x\in Z\ \Leftrightarrow\ V[H]\models\varphi_{Z}(x)
\]
for all $x\in V[H]$. Next we let $\dot{y}$ be a $\mathbb{C}_{\alpha}$-name
such that $\dot{y}_{G\restriction\alpha}=y$. Then we may put things
together and fix $p\in\mathbb{C}_{\alpha}$ such that: 
\[
p\Vdash\varphi_{Z}(\dot{y})\wedge\dot{y}\notin\check{V}\wedge|\{x\subseteq\lambda\ |\ \varphi_{Z}(x)\}|\leq\omega.
\]
Using the first two conjuncts and noting that $y\subseteq\lambda\subseteq V$,
it can be seen -- using Solovay's argument -- that in $V[G]$, $Z$
contains a perfect set. But using the final conjunct, we see that
$Z$ is countable in $V[G]$, which gives us our desired contradiction.
\end{proof}
We now recall an important fact about Steel's generic multiverse.
\begin{fact}
(Goldberg, Lemma 38 in \citep{MeadMadPGMV}) There is a total computable
function $e:\mathcal{L}_{MV}\to\mathcal{L}_{\in}(\dot{M})$ with the
following property. Suppose $\mathcal{W}\models MV$, $M$ is a world
of $\mathcal{W}$, and $N$ is the collection of sets of $\mathcal{W}$.
Then\label{fact:Goldberg1} 
\[
\mathcal{W}\models\varphi(\bar{x})\Leftrightarrow\langle N,M\rangle\models e(\varphi)(\bar{x})
\]
where $\langle N,M\rangle$ denote a model of $\mathcal{L}_{\in}$
expanded by a one-place relation symbol.
\end{fact}

We now use this fact to prove a special case of World-domination
where only one world parameter is involved. 
\begin{lem}
$MV$ proves the following. Suppose $M$ is an inner model of $ZFC$
that is definable using a single world parameter $U$ and parameters
from $W$. Then $M\subseteq U$.\label{lem:Goldberg2}
\end{lem}

\begin{proof}
Let $\mathcal{W}$ be a model of $MV$ and without loss of generality,
we suppose that $\mathcal{W}$ is countable. We show that the claim
holds in $\mathcal{W}$. We start by letting $\varphi_{M}(x,\bar{y},U)$
be the formula defining $M$ in $\mathcal{W}$ where $U$ is a world
parameter $\bar{y}$ is a finite sequence of set parameters from $U$.
Now using Fact \ref{fact:Steel comp}, we may fix a $\mathbb{C}_{Ord^{U}}$-generic
$G$ such that $U[G]$ is identical to the set domain of $\mathcal{W}$.
Then using Fact \ref{fact:Goldberg1}, we see that 
\[
\mathcal{W}\models\varphi_{M}(x,\bar{y},U)\ \Leftrightarrow\ \langle U[G],U\rangle\models e(\varphi)(x,\bar{y}).
\]
Thus, we see that $M$ is defined in $U[G]$ by a formula of $\mathcal{L}_{\in}$
using $U$ as a parameter and parameters from $U$. Finally, we use
Lemma \ref{lem:warmup} to see that $M\subseteq U$, as required.
\end{proof}
Next, we use Amalgamation to prove the general case where multiple
world parameters are involved. 
\begin{thm}
$MV$ proves the following. Suppose $M$ is an inner model of $ZFC$
that is definable using set and world parameters. Then there is some
world $W$ such that $M\subseteq W$.\label{thm:Goldberg3}
\end{thm}

\begin{proof}
Again, we show that the claim holds in arbitrary models of $MV$.
We start by fixing a model $\mathcal{W}$ of $MV$ and we shall now
work inside $\mathcal{W}$. Our plan is to use the Amalgamation axiom
to reduce our problem to that addressed in the previous lemma. We
suppose that $M$ is definable by the $\mathcal{L}_{MV}$-formula
$\varphi(x,y_{0},...,y_{m},W_{0},...,W_{n},)$ where $y_{0},...,y_{m}$
are set parameters and $W_{0},...,W_{n}$ are world parameters. First,
we let $W_{y_{i}}$ be a world with $y_{i}\in W_{y_{i}}$ for all
$i\leq m$,. Then by repeated applications of the Amalgamation axiom,
we may fix a world $W$ such that for all $j\leq n$, $W_{j}\subseteq W$
and for all $i\leq m$, $W_{y_{i}}\in W$. The observe by the Laver-Woodin
theorem, each of these worlds is definable in $W$ from a parameter
in $W$. Thus, $M$ is definable by a formula using $W$ and elements
of $W$. Thus by Lemma \ref{lem:Goldberg2}, we see that $M\subseteq W$
as required.
\end{proof}
Finally, our main claim is complete.
\begin{cor}
$MV\vdash\text{World-domination}$.
\end{cor}

\subsection{A couple of applications in relative interpretablity}

Before we move on to discuss these results, we first take a slight
detour into the world of interpretation. As has been noted by \citet{EnayatLelykFOCat},
categoricity properties are often closely related to important equivalence
relations between theories. We shall see that $MV$ bears further
witness to this close relationship. The following quite technical
definition will allow us to set things up. For models $\mathcal{A}$
and $\mathcal{B}$ of some language $\mathcal{L}$ we shall write
$\mathcal{A}\triangleright_{par}\mathcal{B}$ to indicate that $\mathcal{A}$
interprets $\mathcal{B}$ with parameters.\footnote{See \citep{EnayatLelykFOCat} for a proper definition, but the essential
idea is that the domain and relations on $\mathcal{B}$ are all definable
in $\mathcal{A}$. } 
\begin{defn}
Let $T$ be a theory articulated in some language $\mathcal{L}$.
We say:\footnote{I believe this terminology with the exception of (3) originates with
\citep{EnayatVisTheme}. I include (3) as a weakening since it provides
an interesting warmup for the main claims.} 
\begin{enumerate}
\item $T$ is \emph{solid} if whenever $\mathcal{A},\mathcal{B}$ and $\mathcal{C}$
are models of $T$ where $\mathcal{A}\triangleright_{par}\mathcal{B}\triangleright_{par}\mathcal{C}$
and $\sigma:\mathcal{A}\cong\mathcal{C}$ is definable in $\mathcal{A}$
from parameters, then there is $\tau:\mathcal{A}\cong\mathcal{B}$
that is also definable in $\mathcal{A}$ from parameters. 
\item $T$ is \emph{tight} if whenever $T_{0}$ and $T_{1}$ are bi-interpretable
extensions of $T$, then $T_{0}$ and $T_{1}$ are the same theory
in the sense that they have the same deductive closure. 
\item $T$ is \emph{weakly tight }if whenever $T_{0}$ and $T_{1}$ are
definitionally equivalent extensions of $T$, then $T_{0}$ and $T_{1}$
are the same theory in the sense that they have the same deductive
closure. 
\end{enumerate}
\end{defn}

(1) implies (2) and (2) implies (3). \footnote{See Remark 9 in \citep{EnayatLelykFOCat}.}
(1) is a technical, but very useful property. (2) and (3) seem to
tell us something about the ``stability'' of a theory. Bi-interpretability
and definitional equivalence are canonical equivalence relations that
give us plausible accounts of when two theories might considered intuitively
equivalent for some purposes. (2) and (3) then tell us that -- something
like -- any pair of intuitively equivalent theories are actually
identical. For some examples, we've already seen above that both $ZFC$
and $PA$ are internally categorical. It can be seen using very similar
arguments that they are also both solid. With regard to $MV$ we can
use what we already have above to show that it is weakly tight.
\begin{cor}
$MV$ is weakly tight; i.e., If $T$ and $S$ are definitionally equivalent
extensions of $MV$, then $T$ and $S$ are the same theory in the
sense that $T\dashv\vdash S$.\label{cor:weakTight}
\end{cor}

\begin{proof}
Suppose that $T$ and $S$ are extensions of $MV$ that are definitionally
equivalent. To make things simpler, suppose that $T$ is articulated
using $\in_{0}$ and $S$ is articulated using $\in_{1}$. Then there
are definitional expansions $T^{+}$ and $S^{+}$ that add definitions
of $\in_{1}$ and $\in_{0}$ respectively to $T$ and $S$ such that
$T^{+}$ and $S^{+}$ have the same deductive closure.\footnote{Here we are using the characterization of definitional equivalence
that \citet{LefeverDefEqNonDisjLang} call \emph{definitional mergeability}.
While this is not always the same as definitional equivalence, it
is the same when the languages of the theories being compared are
disjoint. Hence, our move to $\in_{0}$ and $\in_{1}$.} Call the deductive closure of these theories $U$. Since we can define
$\in_{0}$ in terms of $\in_{1}$ and vice versa, this means that
$U$ extends $MV^{*}(\in_{0},\in_{1})$. Thus, by Theorem \ref{thm:MV*intCat},
there is a definable isomorphism between the multiverses associated
with $\in_{0}$ and $\in_{1}$. This means that in any model $\mathcal{A}$
of $U$, the reducts $\mathcal{A}\restriction\in_{0}$ and $\mathcal{A}\restriction\in_{1}$
of $\mathcal{A}$ to $\in_{0}$ and $\in_{1}$ respectively are elementary
equivalent. Thus, since $T^{+}$ and $S^{+}$ are conservative extensions
of $T$ and $S$, we see that they have the same theorems.
\end{proof}
In many interesting cases, tightness follows from weak-tightness.
For some examples, \citet{VisserFriedBitoSyn} have shown that for
a wide class of theories, including $ZFC$, bi-interpretability implies
definitional equivalence. More generally, \citet{EnayatLelykFOCat}
have shown that for sequential theories that eliminate imaginaries,
internal categoricity implies tightness.\footnote{See Proposition 50 in \citep{EnayatLelykFOCat}. See Section 4.4 of
\citep{Hodges} for a discussion of imaginary elimination. A sequential
theory is defined in \citep{VisserFriedBitoSyn}, however, one might
think of it as a particularly low bar for a foundational theory to
meet and one that every theory in this paper traverses.} $ZFC$ is such a theory. When a theory eliminates imaginaries, this
means (roughly) that it can deliver representatives for any definable
equivalence relation. In $ZFC$ this is easily achieved by using Scott's
trick by taking the set of elements of minimal rank from each cell
of the equivalence relation. This technique crucially relies on the
powerset axiom to deliver the relevant $V_{\alpha}$ to intersect.
A such, it is not available in $ZFC_{count}^{-}$. Nonetheless, it
is possible to show that $ZFC_{count}^{-}$ is solid as we now do.
The original proof of this claim achieves is goal by drawing it as
a corollary of the facts that: second order arithmetic is bi-interpretable
with $ZFC_{count}^{-}$; second order arithmetic is tight; and tightness
is preserved through bi-interpretations. The proof below is a little
more direct and also harks back to our proof of Theorem m\ref{thm:ZFCcount}.
Thus, it seems worthwhile to include it.
\begin{thm}
\citep{EnayatVisTheme} $ZFC_{count}^{-}$ is solid.\label{thm:ZFC-solid}
\end{thm}

\begin{proof}
Suppose that $\mathcal{A}$, $\mathcal{B}$ and $\mathcal{C}$ are
models of $ZFC_{count}^{-}$ such that $\mathcal{A}\triangleright_{par}\mathcal{B}\triangleright_{par}\mathcal{C}$
and there is some $\sigma:\mathcal{A}\cong\mathcal{C}$ that is definable
in $\mathcal{A}$ from parameters. We claim that $\mathcal{A}$ can
define with parameters some $\tau:\mathcal{A}\cong\mathcal{B}$. The
argument is similar to our proof that $ZFC_{count}^{-}$ is internally
categorical. Since $\mathcal{C}$ is definable in $\mathcal{B}$ we
can take the transitive collapse $\mathcal{C}_{1}$ of $\mathcal{B}$
in $\mathcal{C}$ via the possibly partial function $\pi_{0}:\mathcal{C}\rightharpoondown\mathcal{C}_{1}$.
At present, we don't know that $\mathcal{C}$ and $\mathcal{C}_{1}$
are isomorphic since we don't know that $\mathcal{C}$ is set-like
in $\mathcal{B}$. Similarly, we can take the transitive collapse
$\mathcal{B}_{1}$ of $\mathcal{B}$ in $\mathcal{A}$ by the possibly
partial function $\pi_{1}:\mathcal{B}\to\mathcal{B}_{1}$. Finally,
wee may also use $\pi_{1}$ to collapse $\mathcal{C}_{1}$ in $\mathcal{A}$
to give us $\mathcal{C}_{2}$. The following diagram summarizes the
situation.
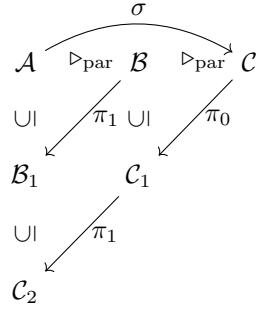
\begin{center}

\begin{tikzpicture}[node distance=1cm and 1cm]     
\node (W) {$\mathcal{A}$};     
\node (U) [right=of W] {$\mathcal{B}$};     
\node (V) [right=of U] {$\mathcal{C}$};          

\node (par1) [right=0.2cm of W] {$\triangleright_{\text{par}}$};     
\node (par2) [right=0.2cm of U] {$\triangleright_{\text{par}}$};          

\node (U1) [below=of W] {$\mathcal{B}_1$};     
\node (V1) [below=of U] {$\mathcal{C}_1$};          

\node (V2) [below=of U1] {$\mathcal{C}_2$};          

\draw[->] (U) -- (U1) node[midway, right] {$\pi_1$};     
\draw[->] (V) -- (V1) node[midway, right] {$\pi_0$};     
\draw[->] (V1) -- (V2) node[midway, right] {$\pi_1$};

\draw[->] (W) to [out=30, in=150] node[midway, above] {$\sigma$} (V);

\node [rotate=90] at ($(U1)!0.5!(W)$) {$\subseteq$};     
\node [rotate=90] at ($(V2)!0.5!(U1)$) {$\subseteq$};
\node [rotate=90] at ($(V1)!0.5!(U)$) {$\subseteq$};

\end{tikzpicture}

\end{center}

We claim that $\mathcal{A}=\mathcal{B}_{1}=\mathcal{C}_{2}$. To see
this, we show that it suffices -- as in the proof of Theorem \ref{thm:ZFCcount}
-- to observe that $\mathcal{A}$ and $\mathcal{C}_{2}$ have the
same sets of natural numbers. We already can see that $\mathcal{C}_{2}\subseteq\mathcal{A}$,
so suppose that $\mathcal{A}\models x\subseteq\omega$. Then we see
that $\mathcal{C}\models\sigma(x)\subseteq\omega$ since $\sigma(\omega^{\mathcal{A}})=\omega^{\mathcal{C}}$.
It is also clear that the transitive closure of $\{\sigma(x)\}$ in
$\mathcal{C}$ is set-like in $\mathcal{C}$ and so we have $\mathcal{B}\models\pi_{0}\circ\sigma(x)\subseteq\omega$.
By similar reasoning we see that $\mathcal{A}\models\pi_{1}\circ\pi_{0}\circ\sigma(x)\subseteq\omega$.
Moreover, it can be seen that $\pi_{1}\circ\pi_{0}\circ\sigma(z)=z$
for any $z$ where the left hand side is defined. Thus, $x=\pi_{1}\circ\pi_{0}\circ\sigma(x)\in\mathcal{C}_{2}$
as required. Now since $\mathcal{A}=\mathcal{B}_{1}$, we see that
$\pi_{1}:\mathcal{B}\cong\mathcal{B}_{1}=\mathcal{A}$ and so $\tau=\pi_{1}^{-1}$
is the isomorphism we've been looking for.
\end{proof}
As in Section \ref{sec:When--is}, we can then piggyback on this argument
to show that $MV$ is also solid.
\begin{thm}
$MV$ is solid. 
\end{thm}

\begin{proof}
Suppose that $\mathcal{W}$, $\mathcal{U}$ and $\mathcal{V}$ are
models of $ZFC_{count}^{-}$ such that $\mathcal{W}\triangleright_{par}\mathcal{U}\triangleright_{par}\mathcal{V}$
and there is some $\sigma:\mathcal{W}\cong\mathcal{V}$ that is definable
in $\mathcal{W}$ from parameters. We claim that $\mathcal{W}$ can
define with parameters some $\tau:\mathcal{W}\cong\mathcal{U}$. As
in the previous proof we take collapses of $\mathcal{V}$ in $\mathcal{U}$
and $\mathcal{U}$ in $\mathcal{V}$ to give us models $\mathcal{V}_{1}$,
$\mathcal{U}_{1}$ and $\mathcal{V}_{2}$ as depicted in the following
diagram.
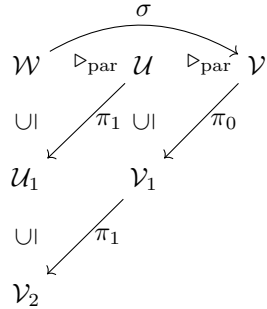
\begin{center}

\begin{tikzpicture}[node distance=1cm and 1cm]     
\node (W) {$\mathcal{W}$};     
\node (U) [right=of W] {$\mathcal{U}$};     
\node (V) [right=of U] {$\mathcal{V}$};          

\node (par1) [right=0.2cm of W] {$\triangleright_{\text{par}}$};     
\node (par2) [right=0.2cm of U] {$\triangleright_{\text{par}}$};          

\node (U1) [below=of W] {$\mathcal{U}_1$};     
\node (V1) [below=of U] {$\mathcal{V}_1$};          

\node (V2) [below=of U1] {$\mathcal{V}_2$};          

\draw[->] (U) -- (U1) node[midway, right] {$\pi_1$};     
\draw[->] (V) -- (V1) node[midway, right] {$\pi_0$};     
\draw[->] (V1) -- (V2) node[midway, right] {$\pi_1$};

\draw[->] (W) to [out=30, in=150] node[midway, above] {$\sigma$} (V);

\node [rotate=90] at ($(U1)!0.5!(W)$) {$\subseteq$};     
\node [rotate=90] at ($(V2)!0.5!(U1)$) {$\subseteq$};
\node [rotate=90] at ($(V1)!0.5!(U)$) {$\subseteq$};

\end{tikzpicture}

\end{center}

Given that $MV$ implies $ZFC_{count}^{-}$, we see that $\pi_{1}:\langle\mathbb{D}^{\mathcal{U}},\in^{\mathcal{U}}\rangle\cong\langle\mathbb{D}^{\mathcal{W}},\in^{\mathcal{W}}\rangle$
is an isomorphism on the set domains of $\mathcal{U}$ and $\mathcal{W}$
that is definable in $\mathcal{U}$. To address the worlds, it will
suffices to show that $\mathcal{W}$ and $\mathcal{U}_{1}$ have the
same worlds. Our argument is similar to the one we used to prove Theorem
\ref{thm:MV*intCat}. First, suppose that $W$ is a $\mathcal{W}$-world.
Then since $\pi_{1}\circ\pi_{0}\circ\sigma$ is the identity function
on sets in $\mathcal{W}$, we see that $\pi_{1}\circ\pi_{0}\circ\sigma``W=W\in\mathcal{V}_{2}\subseteq\mathcal{U}_{1}$.
Thus by World-domination in $\mathcal{U}_{1}$, we may fix a $\mathcal{U}_{1}$-world
$U\supseteq W$. And then by World-domination in $\mathcal{W}$ we
may fix a $\mathcal{W}$-world $W^{*}\supseteq U$. Thus, we have
$W\subseteq U\subseteq W^{*}$ where $W^{*}$is a generic extension
of $W$ and so $W$ is a generic refinement of a $\mathcal{U}_{1}$-world
meaning that $W$ is also a $\mathcal{U}_{1}$-world. A similar argument
establishes that every $\mathcal{U}_{1}$ world is also a $\mathcal{W}$
world. 
\end{proof}
With this in hand, we then see (roughly) that equivalent extensions
of $MV$ (and $ZFC_{count}^{-}$) are, in fact, identical. 
\begin{cor}
$ZFC_{count}^{-}$ and $MV$ are both tight; i.e., any pair of theories
extending $MV$ (or $ZFC_{count}^{-}$) that are bi-interpretable
have the same deductive closure.
\end{cor}

\section{Discussion\label{sec:Discussion}}

I think the results above are sufficiently surprising, and perhaps
confusing, that some discussion is warranted. However, I'll concede
at the get-go that while I hope to clarify matters, I don't have anything
definitive to say on the philosophy: just more questions. Nonetheless,
from a logical point of view I think these results are of obvious
interest and this justifies further investigation. Categoricity results
related to second order logic (very broadly construed) tend to hold
for theories whose models have an obvious ``inductive structure.''\footnote{I'm mentioning second order logic here to distinguish these results
from other forms of categoricity, like the categoricity of the theory
of dense linear orders over the space of countable models. } The natural examples thus come from arithmetic and set theory, where
we start from a basic object like $0$ or $\emptyset$ and then construct
the rest of the universe inductively from there. The key point is
that one level of the construction determines the next. But $MV$
is only partially like this. Since the set domain of a model of $MV$
satisfies $ZFC_{count}^{-}$, we have some inductive structure from
which something like a standard categoricity argument can be conducted.
But nothing like this comes along with the world domain. The worlds
are generated using generic extension (and refinement) through a process
that is not inductive. Given a particular world, it simply makes no
sense to talk about the next world. They are not ordered in an interesting,
well-founded manner. To prove the internal categoricity of $MV^{*}$
we take up a different strategy that relies on the ways in which the
worlds of the respective multiverse overlap each other when collapsed.
This makes the proof of the internal categoricity of $MV^{*}$ quite
different from other proofs of its kind. At present, I know of no
other example like it. It would be good to know if other examples
are out there. 

Let us now move onto the philosophy. I'll start by considering how
the internal categoricity result above might affect our interpretation
of similar claims. Then I'll consider some reasons we might have for
disregarding the significance of those considerations. 

In the introduction to this paper we observed that traditional, external
categoricity results for set theory rely on the use of full models
of second order $ZFC$, which presuppose that it makes sense to talk
of the ``true'' powerset operation in order to define those models.
Given that large chunk of the burden in the categoricity claim is
obtaining a unique powerset operation inside those models, we appear
to be presupposing what we set out to obtain. This has been a classic
reason to disregard the evidential value of these results when considering
questions about the determinateness of the set theory's subject matter.
However, internal categoricity makes no use of such second order resources.
The theorem and the proof are all articulated in first order logic,
within a theory $ZFC(\in_{0},\in_{1})$ that shares its membership
relations evenly amongst its axiom schemata. Moreover loosely following
\citet{ParsonNatNum}, we can justify the use of such a theory using
a linguistically motivated thought experiment that is closely related
to the way in which we use translation to interpret different extensions
of $ZFC$. The theorem then shows that the universes associated with
$\in_{0}$ and $\in_{1}$ are isomorphic and that this isomorphism
is visible within the context of the thought experiment since it is
definable. If we then consider any statement articulated in $\mathcal{L}_{\in}$,
like the continuum hypothesis, we see that both the $\in_{0}$-user
and the $\in_{1}$ -user must agree on its truth value (even if they
are not in a position to decide it). Thus, we seem to have some reason
to think that the continuum hypothesis has a determinate truth value
using an approach that is not exposed to circularity problems but
is based in a methodology common to the practice of set theorists.
Now this is hardly a watertight argument, but it does look like a
promising sketch. 

Unfortunately, I think the internal categoricity of $MV^{*}$ suggests
that this sketch cannot be filled in successfully. Or at least, any
effort to fill it in will need to explain what is happening with $MV^{*}$
and this will likely make the completed argument subtle and complex.
This seems unattractive as one of the draws of using categoricity
arguments for these purposes is the promise of extracting philosophy
straightforwardly from the math. While this kind of strategy rarely
works, when it does, the resulting claim is difficult to challenge
when it is closely related to something incontrovertible: a theorem.
But once the philosophy creeps in, the wiggle room quickly expands
and the opportunities to disagree multiply. 

So why can't the argument sketch be filled in? The essential idea
is that people using $MV$ have access to an analogous argument that
would appear to lead to a conclusion that contradicts the one drawn
above. This suggests that something is wrong with the argument sketch.
Let's try to step through this analogy patiently. We saw above that
$MV$ was introduced as an alternative to $ZFC$ that takes the incompleteness
wrought by forcing as a feature and not a bug. It was designed to
be an alternative foundation for mathematics that is more conservative
in the questions it permits us to ask. Regarding $ZFC$ and its canonical
extensions, Steel remarks, ``the mathematical theory based on large
cardinal hypotheses that we have produced to date can naturally be
expressed in the multiverse language'' \citep{SteelGP}. Moreover,
it has been shown that $ZFC$ and $MV$ are not only equiconsistent,
but there are interpretations witnessing this that display a degree
of invertibility akin to bi-intepretability, although somewhat weaker.\footnote{See Theorem 18 in \citep{Meadows2Args} and Proposition 6.1.1 in \citep{SteelOnMadMead}
for more detail. In particular, Meadows shows that $ZFC+V=HOD+GA$
is sententially equivalent to $MV+\exists W(V=HOD+GA)^{W}$. This
means that there are interpretations going back and forth between
these theories that take models of one theory and return a model that
is elementary equivalent to the original.} For the purposes of providing a foundation for mathematics and, indeed,
doing contemporary set theory $ZFC$ and $MV$ are functionally identical.
But there is, of course, a salient difference in their treatments
of the continuum hypothesis. While in the framework of $ZFC$, the
continuum hypothesis is naturally articulated as a simple statement
about a putative bijection between two sets, there is no corresponding
formulation in $MV$. From the local perspective, it is true in some
worlds and false in others. From the global perspective, it simply
cannot be articulated. This is good reason to think that $MV$ is
successful in satisfying the goal of providing a foundation for mathematics
that takes deep incompleteness seriously. 

Nonetheless, the users of $MV$ (like those of $ZFC$) might still
have some qualms about the the fixedness of the content of their new
theory. Perhaps $MV$ is an algebraic theory like group theory that
has many very different models. I do not want to make any claims about
whether this would affect $MV$'s ability to provide a foundation
for mathematics, but I think we can set that problem aside. Regardless
of how it is settled, it would still be interesting and informative
for users of $MV$ to understand whether it is algebraic or not. This
is a natural time for internal categoricity to enter our story. We
might follow something like the Parsonian thought experiment described
above for $ZFC$, with $MV^{*}$ in its place.\footnote{The reader will note that I'm moving quite freely between $MV$ and
$MV^{*}$ here. This will be discussed further below.} We might imagine two users of $MV^{*}$ who wonder whether they are
talking about the same thing. Following the template, we might then
move to a combined language with two membership relations $\in_{0}$
and $\in_{1}$. Regarding the theory used to govern those relations,
we might agree to use the combined language in each of the axiom schemata,
thus, giving us $MV^{*}(\in_{0},\in_{1})$. The results above then
tell us that there will be a definable isomorphism between the multiverse
associated with $\in_{0}$ and that of $\in_{1}$. This seems to provide
a helpful response to qualms about the fixedness of $MV$'s content.
For example, it is obvious that group theory cannot be axiomatized
in a way that is is internally categorical.\footnote{Perhaps the easiest way to see this is to observe that group theory
has, and is intended to have, distinct models with different cardinalities.} Thus, we see something about $MV$ that makes it more like $ZFC$
than group theory. So far so good. But can we say more? I think that
a more extensive analysis of the thought experiment sketched above
could provide us with further insights that would shed light on both
$MV$ and internal categoricity. Let us, however, leave that for future
work and return to our analogy. 

Earlier we saw our, arguably straw, $ZFC$ users take the solace drawn
from the internal categoricity of $ZFC$ and use that to sketch an
argument toward a bold conclusion about the determinateness of the
continuum hypothesis. Could incautious $MV$ users do something similar?
Given that we just have a coarse-grained sketch, I don't see why they
couldn't. The $MV$ users have a theory that is compatible with the
idea that the continuum hypothesis is making something like a category
error. Internal categoricity then appears to tell us something positive
about the semantic stability of the perspective they have adopted.
Any pair of $MV^{*}$ users who commit to the combined theory $MV^{*}(\in_{0},\in_{1})$
will also be committed to the continuum hypothesis being indeterminate
from both the local and global perspectives. At this point, I'm reluctant
to follow through on the analogy and say that this puts them in a
position to argue that the continuum hypothesis is therefore, indeterminate.
But I think this is a matter of hindsight since I now feel the same
way about the $ZFC$ users described above. Whatever conclusions can
be drawn from the internal categoricity of these theories, it cannot
tell us that the continuum hypothesis is both determinate and indeterminate.
Nonetheless, it might be put to work to answer a coherence challenge.
We might say that the internal categoricity of $MV$ does put us in
a position to sustain an attitude that continuum hypothesis is indeterminate.

With that little argument out of the way, let's close by considering
a couple of concerns one might have about this line of reasoning.
As we go, I'll also offer brief responses to each of them. For our
first worry, the reader may have observed that I moved quite freely
from talking about $MV^{*}$ to $MV$ in the previous paragraphs.
These axiomatizations have an important difference. $MV^{*}$ is internally
categorical, while $MV$ is not. Moreover, given that $MV$ is the
official axiomatization, we might wonder whether we should care about
what $MV^{*}$ and its peculiar properties. This might give us reason
to doubt the robustness of the analogy described above. Our response
is relatively straightforward. While $ZFC$ is internally categorical,
it also has an alternative axiomatization $ZFC^{*}$ that is not internally
categorical.\footnote{See Theorem 52 in \citep{EnayatLelykFOCat} for a proof of this.}
Thus, exactly the same problem emerges in the context of $ZFC$. So
which axiomatization is the right one? This is an interesting question
that has become more pressing in recent years.\footnote{See the conclusion of \citep{EnayatLelykFOCat} for some discussion
of this.} But it is not so easy to answer. In contrast to the convention adopted
in this paper, it is more common to regard a theory as the deductive
closure of some set of axioms rather than the axioms themselves. Moreover,
this accords well with the ordinary practice of set theorists. We
are generally just interested in what follows from $ZFC$ regardless
of how it is axiomatized. Evidence for this can be found in the fact
that different textbooks frequently give quite different axiomatizations
of $ZFC$. Some use Collection rather than Replacement. Some prefer
Foundation over Set Induction.\footnote{For another classic difference, Kunen delivers a Pairing Axiom by
saying that given any sets $x$ and $y$ there is a set $z$ that
contains $x$ and $y$, while Jech says more by saying that there
is a set $w$ that contains \emph{exactly }$x$ and $y$ \citep{KunenST,JechST}.
Of course, Separation and Extensionality wipe away any worries of
inequivalence, but the point is just that there are many distinct
axiomatizations of $ZFC$ out there and we rarely have need to worry
about this. I should also note that when we consider weakenings of
$ZFC$ then the question of axiomatization can become very important.
For example, it is well known that in the absence of Extensionality
or Powerset, Replacement behaves very oddly while Collection does
not \citep{ScottExtMore,GitmanZFC-}.} But none of this changes the theorems of $ZFC$, so the mathematical
practice of users of $ZFC$ remains generally unaffected. Perhaps
the right question to ask is not: is this axiomatization of theory
$T$ internally categorical; but rather, is there an axiomatization
of theory $T$ that is categorical?\footnote{Of course, another option would be to consign the internal categoricity
of both $ZFC$ and $MV^{*}$ to the flames. Perhaps the axiom sensitivity
of internal categoricity makes it insignificant. I think it would
be extremely premature to draw such a conclusion when so much remains
unknown. Moreover, we've seen, with Corollary \ref{cor:weakTight},
that despite the axiom sensitivity, the internal categoricity of $MV^{*}$
can still be used to draw significant conclusions about $MV$ even
when it is construed as a deductively closed theory.} If so, $MV^{*}$'s internal categoricity seems to be just a significant
as that of $ZFC$. If not, then much more work needs to be done to
understand when one axiomatization of a theory is preferable to another.
Cards on the table, I think this is a fertile problem to explore that
is likely to deliver significant foundational insights that could
be pertinent to our current problem. But at present, this is a field
in its infancy. 

For our second and final concern, let's consider the World-domination
schema. In $MV(\in_{0},\in_{1})$ both $\in_{0}$ and $\in_{1}$ are
permitted into both the $\in_{0}$ and $\in_{1}$ versions of the
World-domination schema. This is directly analogous to what we did
with $ZFC(\in_{0},\in_{1})$. However, in the $ZFC(\in_{0},\in_{1})$
case we also spent some time discussing why we should allow the both
membership relations into both versions of, say, the Replacement Schema.
We are yet to do this for World-domination. And there is an important
hitch. We motivated our use of the shared version of Replacement on
the grounds that it provided a natural approximation of the full second
order version of Replacement. No such motivation can be offered for
World-domination. To see this recall that World-domination says that
every \emph{definable} inner model is covered by a world. What would
it mean to say that this was an approximation of the full second order
version of World-domination? To answer this question, we need to know
what the full second order version of World-domination would look
like. The obvious suggestion is the following: every inner model is
covered by a world. With Replacement we moved from \emph{definable}
class functions to \emph{arbitrary} class functions; here we move
from \emph{definable} inner models to \emph{arbitrary} inner models.
But second order World-domination is inconsistent.\footnote{It is also worth bearing in mind that the other axiom schemata also
face similar problems. For example, if we were to treat the Replacement
Schemata relativized to worlds as an approximation of a second order
version this also leads to inconsistency. To see this suppose $W$
is a world and $W[c]$ is a generic extension of $W$ by a Cohen real
$c\subseteq\omega$. Now let $f:\omega\to\omega$ be the injection
that enumerates $c$. This is a(n improper) class function over $W$,
but a second order Replacement Axiom would tell us that $f``\omega=c\in W$,
which is impossible.} To see this consider, the models $\mathcal{W}_{0}$ and $\mathcal{W}_{1}$
described in the proof of Theorem \ref{thm:There-exist-models} (showing
that $MV$ is not internally categorical). We claim that $\mathcal{W}_{0}$
doesn't satisfy second order World-domination. To see this recall
that while $\mathcal{W}_{0}$ is generated from a world $M$, $\mathcal{W}_{1}$
is generated from a class generic $N$ extension of $M$ that no world
in $\mathcal{W}_{0}$ can cover. As such, we have in $N$ an uncovered
inner model of $\mathcal{W}_{0}$. So how do we motivate World-domination
and how does this affect the analogy above? Our response is to concede
that World-domination is not an approximation of a second order axiom,
but rather, it is about definability. Theorem \ref{thm:Goldberg3}
not only tells us that World-domination is a consequence of $MV$,
its proof also gives us some explanation as to why it is true and
some expectation of when it will continue to hold. In particular,
we see that World-domination reveals deep limitations in what one
can expect to define over a model of $MV$. In allowing both membership
relations into the World-domination Schemata, we can be understood
as merely flagging our expectation that these limitations will continue
to hold. Of course, this is not as simple as saying that we are approximating
a second order axiom, but this may be more of a blessing than a curse.
When we motivated shared Replacement with the approximation story,
we might wonder how far we've removed ourselves from second order
logic and our worries about circularity. When motivating World-domination
through definability rather than second order logic, there is no implicit
appeal to the collection of all subclasses of the domain. Perhaps
this is a benefit. I think it's fair to say that more work is required
here, and that this work will tell much more about internal categoricity
and foundational theories in mathematics. But even at just the beginning
of this road, it is interesting to see an axiom schema in a foundational
setting that cannot be explained as an approximation of a second order
cousin.

\subsection*{Conclusion}

In this paper, we've shown that Steel's multiverse theory can be axiomatized
in alternative manners such that it is both internally categorical
and not internally categorical. This makes Steel's theory an interesting
example to consider in the theory of relative interpretation. Given
that internal categoricity is sometimes thought to have significance
with regard to the determinateness of the continuum hypothesis and
that $MV$ is designed to, in a sense, ignore it, this also makes
the internal categoricity of (a version of) $MV$ interesting in the
philosophy and foundations of mathematics. We have argued that this
result suggests that arguments from internal categoricity to the fixedness
of subject matter in a theory are unlikely to work or, at least, will
be much more complicated than one might have expected.

\bibliographystyle{plainnat}

\begin{thebibliography}{45}
\providecommand{\natexlab}[1]{#1}
\providecommand{\url}[1]{\texttt{#1}}
\expandafter\ifx\csname urlstyle\endcsname\relax
  \providecommand{\doi}[1]{doi: #1}\else
  \providecommand{\doi}{doi: \begingroup \urlstyle{rm}\Url}\fi

\bibitem[Bagaria and Poveda(2023)]{BagariaExtPres}
Joan Bagaria and Alejandro Poveda.
\newblock More on the preservation of large cardinals under class forcing.
\newblock \emph{Journal of Symbolic Logic}, 88\penalty0 (1):\penalty0 290--323,
  2023.

\bibitem[Bagaria and Ternullo(2023)]{BAGARIA_TERNULLO_2023}
Joan Bagaria and Claudio Ternullo.
\newblock Steel's programme: Evidential framework, the core and ultimate-{L}.
\newblock \emph{The Review of Symbolic Logic}, 16\penalty0 (3):\penalty0
  788--812, 2023.

\bibitem[Blue(2024)]{BLUE_2024}
Douglas Blue.
\newblock The generic multiverse is not going away.
\newblock \emph{The Review of Symbolic Logic}, pages 1--33, 2024.
\newblock \doi{10.1017/S1755020324000297}.

\bibitem[Button and Walsh(2016)]{ButtonWalshCat}
Tim Button and Sean Walsh.
\newblock Structure and categoricity: Determinacy of reference and truth value
  in the philosophy of mathematics.
\newblock \emph{Philosophia Mathematica}, 24\penalty0 (3):\penalty0 283--307,
  2016.

\bibitem[Cohen(1963)]{Cohen1963-COHTIO-5}
Paul Cohen.
\newblock The independence of the continuum hypothesis.
\newblock \emph{Proc. Nat. Acad. Sci. USA}, 50\penalty0 (6):\penalty0
  1143--1148, 1963.

\bibitem[Dedekind(1963)]{Dedekind}
Richard Dedekind.
\newblock \emph{Essays on the Theory of Numbers}.
\newblock Dover Publications, New York, 1963.

\bibitem[Enayat(2016)]{EnayatVisTheme}
Ali Enayat.
\newblock Variations on a {V}isserian theme.
\newblock In Jan van Eijck, Rosalie Iemhoff, and Joost~J. Joosten, editors,
  \emph{Liber Amicorum Alberti: a tribute to Albert Visser}, pages 99--110.
  College Publications, 2016.

\bibitem[Enayat and \L{}e\l{}yk(2024)]{EnayatLelykFOCat}
Ali Enayat and Mateusz \L{}e\l{}yk.
\newblock Categoricity-like properties in the first order realm.
\newblock \emph{Journal for the Philosophy of Mathematics}, 1:\penalty0 63--98,
  Sep. 2024.

\bibitem[Field(1999)]{FieldDeflCons}
Hartry Field.
\newblock Deflating the conservativeness argument.
\newblock \emph{Journal of Philosophy}, 96\penalty0 (10):\penalty0 533--540,
  1999.

\bibitem[Foreman and Kanamori(2009)]{foreman2009handbook}
M.~Foreman and A.~Kanamori.
\newblock \emph{Handbook of Set Theory}.
\newblock Springer Netherlands, 2009.

\bibitem[Gitman et~al.(2016)Gitman, Hamkins, and Johnstone]{GitmanZFC-}
Victoria Gitman, Joel~David Hamkins, and Thomas~A. Johnstone.
\newblock What is the theory {Z}{F}{C} without power set?
\newblock \emph{Mathematical Logic Quarterly}, 62\penalty0 (4-5):\penalty0
  391--406, 2016.

\bibitem[Hamkins(2012)]{HamMult}
Joel~David Hamkins.
\newblock The set-theoretic multiverse.
\newblock \emph{The Review of Symbolic Logic}, 5:\penalty0 416--449, 2012.

\bibitem[Hjorth(1995)]{HjorthU2}
Greg Hjorth.
\newblock The size of the ordinal u2.
\newblock \emph{Journal of the London Mathematical Society}, 52\penalty0
  (3):\penalty0 417--433, 1995.

\bibitem[Hodges(1997)]{Hodges}
Wilfrid Hodges.
\newblock \emph{A Shorter Model Theory}.
\newblock CUP, Cambridge, 1997.

\bibitem[Jech(2003)]{JechST}
Thomas Jech.
\newblock \emph{Set Theory}.
\newblock Springer, Heidelberg, 2003.

\bibitem[Kanamori(2003)]{kanamori2003higher}
A.~Kanamori.
\newblock \emph{The Higher Infinite: Large Cardinals in Set Theory from Their
  Beginnings}.
\newblock Springer, 2003.

\bibitem[Kanovie and Lyubetsky(2017)]{KanovieODrealsinGround}
Vladimir Kanovie and Vassily Lyubetsky.
\newblock Countable {O}{D} sets of reals belong to the ground model.
\newblock \emph{Archive for Mathematical Logiiic}, 57:\penalty0 285--298, 2017.

\bibitem[Koellner(2013)]{KoellnerHAM}
Peter Koellner.
\newblock Hamkins on the multiverse.
\newblock 2013.

\bibitem[Kreisel(1969)]{KreisalComp}
Georg Kreisel.
\newblock Informal rigour and completeness proofs.
\newblock In Jaako Hintikka, editor, \emph{The Philosophy of Mathematics}.
  Oxford University Press, London, 1969.

\bibitem[Kunen(2006)]{KunenST}
Kenneth Kunen.
\newblock \emph{Set Theory: an introduction to independence proofs}.
\newblock Elsevier, Sydney, 2006.

\bibitem[Kunen(2011)]{KunenST2}
Kenneth Kunen.
\newblock \emph{Set Theory}.
\newblock College Publications, London, 2nd edition, 2011.

\bibitem[Larson(2004)]{larsonstationary}
P.B. Larson.
\newblock \emph{The Stationary Tower: Notes on a Course by W. Hugh Woodin}.
\newblock University lecture series. American Mathematical Soc., 2004.

\bibitem[Lefever and Sz\'{e}kely(2019)]{LefeverDefEqNonDisjLang}
Koen Lefever and Gergely Sz\'{e}kely.
\newblock On generalization of definitional equivalence to non-disjoint
  languages.
\newblock \emph{Journal of Philosophical Logic}, 48\penalty0 (4):\penalty0
  709--729, 2019.

\bibitem[Lindstr{\"o}m(2003)]{lindstrm2003aspects}
P.~Lindstr{\"o}m.
\newblock \emph{Aspects of Incompleteness: Lecture Notes in Logic 10}.
\newblock Lecture notes in logic. Taylor \& Francis, 2003.

\bibitem[Maddy and Meadows(2020)]{MeadMadPGMV}
Penelope Maddy and Toby Meadows.
\newblock A reconstruction of {S}teel's multiverse project.
\newblock \emph{Bulletin of Symbolic Logic}, 26\penalty0 (2):\penalty0
  118--169, 2020.

\bibitem[Maddy and V\"{a}\"{a}n\"{a}nen(2023)]{MaddyVaananenCat}
Penelope Maddy and Jouko V\"{a}\"{a}n\"{a}nen.
\newblock \emph{Philosophical Uses of Categoricity Arguments}.
\newblock Elements in the Philosophy of Mathematics. Cambridge University
  Press, 2023.

\bibitem[Meadows(2013)]{MeadowsCat}
Toby Meadows.
\newblock What can a categoricity theorem tell us?
\newblock \emph{The Review of Symbolic Logic}, 6\penalty0 (3):\penalty0
  524--544, 2013.

\bibitem[Meadows(2021)]{Meadows2Args}
Toby Meadows.
\newblock Two arguments against the generic multiverse.
\newblock \emph{Review of Symbolic Logic}, 14\penalty0 (2):\penalty0 347--379,
  2021.

\bibitem[Mostowski(1967)]{MostowskiRRST}
Andrzej Mostowski.
\newblock Recent results in set theory.
\newblock In Imre Lakatos, editor, \emph{Problems in the Philosophy of
  Mathematics}. North Holland Publishing Company, Amsterdam, 1967.

\bibitem[Parsons(1990)]{ParsonNatNum}
Charles Parsons.
\newblock The uniqueness of the natural numbers.
\newblock \emph{Iyyun: The Jerusalem Philosophical Quarterly}, 39:\penalty0
  13--44, 1990.

\bibitem[Reitz(2007)]{ReitzTGA}
Jonas Reitz.
\newblock The ground axiom.
\newblock \emph{Journal of Symbolic Logic}, 72\penalty0 (4):\penalty0
  1299--1317, 2007.

\bibitem[Scott(1961)]{ScottExtMore}
Dana Scott.
\newblock \emph{More on the axiom of extensionality}, pages 115--131.
\newblock Magnes Press, 1961.

\bibitem[Shepherdson(1951)]{ShepInI}
J.~C. Shepherdson.
\newblock Inner models for set theory - part {I}.
\newblock \emph{Journal of Symbolic Logic}, 16\penalty0 (3):\penalty0 161--190,
  1951.

\bibitem[Solovay(1970)]{Solovay1970AMO}
Robert~M. Solovay.
\newblock A model of set-theory in which every set of reals is {L}ebesgue
  measurable*.
\newblock \emph{Annals of Mathematics}, 92:\penalty0 1--56, 1970.

\bibitem[Steel(2024)]{SteelOnMadMead}
John Steel.
\newblock Generically invariant set theory.
\newblock In \emph{The Philosophy of Penelope Maddy}. Springer, 2024.

\bibitem[Steel(2014)]{SteelGP}
John~R. Steel.
\newblock G{\"o}del's program.
\newblock In Juliette Kennedy, editor, \emph{Interpreting G\"o{}del: Critical
  Essays}. Cambridge University Press, 2014.

\bibitem[Usuba(2017)]{UsubaDDGpub}
Toshimichi Usuba.
\newblock The downward directed grounds hypothesis and very large cardinals.
\newblock \emph{Journal of Mathematical Logic}, 17\penalty0 (02):\penalty0
  1750009, 2017.

\bibitem[V\"{a}\"{a}n\"{a}nen(2019)]{VaanZerm}
Jouko V\"{a}\"{a}n\"{a}nen.
\newblock An extension of a theorem of {Z}ermelo.
\newblock \emph{The Bulletin of Symbolic Logic}, 25\penalty0 (2):\penalty0
  208--212, 2019.

\bibitem[V\"{a}\"{a}n\"{a}nen(2021)]{VaananenTrIntCat}
Jouko V\"{a}\"{a}n\"{a}nen.
\newblock Tracing internal categoricity.
\newblock \emph{Theoria}, 87\penalty0 (4):\penalty0 986--1000, 2021.

\bibitem[Visser and Friedman(2014)]{VisserFriedBitoSyn}
Albert Visser and Harvey~M. Friedman.
\newblock When bi-interpretability implies synonymy.
\newblock \emph{Logic Group preprint series}, 320, 2014.

\bibitem[Weston(1976)]{WestCont}
Thomas Weston.
\newblock Kreisel, the continuum hypothesis and second order set theory.
\newblock \emph{The Journal of Philosophical Logic}, 5:\penalty0 281--298,
  1976.

\bibitem[Woodin(2004)]{WoodinSTAR}
W.~Hugh Woodin.
\newblock Set theory after russell; the journey back to eden.
\newblock In G.~Link, editor, \emph{100 Years of Russell's Paradox}. De
  Gruyter, 2004.

\bibitem[Woodin(2011)]{WoodinROI}
W.~Hugh Woodin.
\newblock \emph{The Realm of the Infinite}.
\newblock Cambridge University Press, 2011.

\bibitem[Woodin(2012)]{WoodinGM}
W.~Hugh Woodin.
\newblock \emph{The Continuum Hypothesis, the Generic Multiverse of Sets, and
  the $\Omega$ Conjecture}.
\newblock Cambridge University Press, 2012.

\bibitem[Zermelo(1976)]{Zermelo}
Ernst Zermelo.
\newblock On boundary numbers and domains of sets: new investigations in the
  foundations of set theory.
\newblock In \emph{From Kant to Hilbert: A Source Book in the Foundations of
  Mathematics}. Oxford University Press, 1976.

\end{thebibliography}

\end{document}